\documentclass{amsart}%
\usepackage{amsfonts,  bbm}
\usepackage{psfrag}
\usepackage{amsmath}
\usepackage{amsfonts}
\usepackage{amssymb}
\usepackage{graphicx}%
\newtheorem{theorem}{Theorem}[section]
\theoremstyle{plain}
\newtheorem{corollary}[theorem]{Corollary}
\newtheorem{lemma}[theorem]{Lemma}
\newtheorem{proposition}[theorem]{Proposition}
\theoremstyle{remark}
\newtheorem{remark}[theorem]{Remark}

\numberwithin{equation}{section}

\newcommand{\beq}{\begin{equation}}
\newcommand{\eeq}{\end{equation}}
\newcommand{\bea}{\begin{eqnarray}}
\newcommand{\eea}{\end{eqnarray}}

\begin{document}
\title[Resonances for manifolds hyperbolic near infinity]{Resonances for manifolds hyperbolic near infinity: Optimal Lower Bounds on Order
of Growth}
\author[D.\ Borthwick]{D.\ Borthwick}
\address{Department of Mathematics and Computer Science, Emory University, Atlanta,
Georgia 30322 USA}
\email{davidb@mathcs.emory.edu}
\author[T.\ J.\ Christiansen]{T.\ J.\ Christiansen}
\address{Department of Mathematics , University of Missouri, Columbia, Missouri 65211 USA}
\email{christiansent@missouri.edu}
\author[P.\ D.\ Hislop]{P.\ D.\ Hislop}
\address{Department of Mathematics, University of Kentucky, Lexington, Kentucky
40506-0027, USA}
\email{hislop@ms.uky.edu}
\author[P.\ A.\ Perry]{P.\ A.\ Perry}
\address{Department of Mathematics, University of Kentucky, Lexington, Kentucky
40506-0027, USA}
\email{perry@ms.uky.edu}
\thanks{DB supported in part by NSF grant DMS-0901937}
\thanks{TC supported in part by NSF grants DMS-0500267 and DMS-1001156.}
\thanks{PDH supported in part by NSF grant DMS-0803379.}
\thanks{PP supported in part by NSF grant DMS-0710477.}
\date{June 24, 2010}

\begin{abstract}
Suppose that $(X,g)$ is a conformally compact $(n+1)$-dimensional manifold
that is hyperbolic near infinity in the sense that the sectional curvatures of
$g$ are identically equal to minus one outside of a compact set $K\subset X$.
We prove that the counting function for the resolvent resonances has maximal
order of growth $(n+1)$ generically for such manifolds. This is achieved by
constructing explicit examples of manifolds hyperbolic at infinity for which
the resonance counting function obeys optimal lower bounds.

\end{abstract}
\maketitle
\tableofcontents

\section{Introduction}

\label{sec:intro1}

Resonances are poles of the resolvent for the Laplacian on a non-compact
manifold. Resonances are the natural analogue of the eigenvalues of the
Laplacian on a compact manifold: they are closely related to the classical
geodesic flow, and determine asymptotic behavior of solutions of the wave equation.

A fundamental object of interest is the resonance counting function, $N(r)$,
defined as the number of resonances (counted with appropriate multiplicity) in
a disc of radius $r$ about a chosen fixed point in the complex plane. Upper
bounds on the resonance counting function of the Laplacian on a Riemannian
manifold $(X,g)$ typically take the form $N(r)\leq Cr^{m}$ for large $r$,
where $m=\dim X$. Lower bounds on the resonance counting function (which imply
the existence of the resonances) are typically much harder to prove.

The purpose of this paper is to prove optimal lower bounds on the order of
growth of the resonance counting function for generic metrics in a class of
manifolds hyperbolic near infinity. Here the \emph{order of growth} of a
counting function $N(r)$ is defined to be
\begin{equation}
\rho=\underset{r\rightarrow\infty}{\lim\sup}\,\left(  \frac{\log N(r)}{\log
r}\right)  , \label{eq:order}%
\end{equation}
and we say that the resonance counting function of the Laplacian on a
Riemannian $(X,g)$ with $\dim X=m$ has \emph{maximal order of growth} if
$\rho=m$. If the resonance counting function does not have maximal order of
growth, we will say that $(X,g)$ is \emph{resonance-deficient}. We will prove
that, among compactly supported metric perturbations of a given metric $g_{0}$
in our class, the set of metrics whose resonance counting function has maximal
order of growth is a dense $G_{\delta}$ set or better; the precise formulation
is given in Theorem \ref{thm:main}.

In even dimensions, the nature of the singularity of the wave trace at zero
makes it easy to obtain generic lower bounds on the resonance counting
function. Hence the main challenge lies in the odd-dimensional case. Our work
here draws on two principle sources: first, Sj\"{o}strand and Zworski's
\cite{SZ:1993} construction of an asymptotically Euclidean metric whose
resonance counting function obeys a lower bound of the form $N(r)\geq Cr^{m}$,
and second, the techniques developed by Christiansen \cite{Christiansen:2005},
\cite{Christiansen:2007}, and Christiansen-Hislop \cite{CH:2005} to prove
lower bounds on the resonance counting function for generic potentials and metrics.

Sj\"{o}strand and Zworski constructed their example of an asymptotically
Euclidean metric with many resonances by gluing a large sphere onto Euclidean
space. They exploit the singularity of the wave trace for the Laplacian on the
sphere from its periodic geodesics, and show that this singularity persists
under gluing. Using the Poisson formula for resonances and a Tauberian
argument, they obtain lower bounds on the counting function.

Here we will use elementary propagation estimates for the wave equation
together with a Poisson formula due to Borthwick \cite{Borthwick:2008} to show
that this same gluing construction can be carried out perturbatively on a
large class of manifolds with nontrivial geometry and topology. This class
consists of conformally compact manifolds with constant curvature $-1$ in a
neighborhood of infinity, described in greater detail in what follows. Then,
we will use Christiansen's method to show that, generically within this class,
the counting functions have maximal order of growth.

Christiansen's method was developed in the context of Euclidean scattering. It
requires that the basic objects of scattering theory (the scattering operator
and scattering phase) remain well-behaved under complex perturbations of the
potential or metric, and also requires that at least one potential or metric
in the class has a resonance counting function with maximal order of growth.
Christiansen's method then shows that the same is true for a dense $G_{\delta
}$ set of metrics or potentials. Such results are \textquotedblleft best
possible\textquotedblright\ in the sense that there are known examples where
the resolvent is entire and there are \emph{no} resonances (see
\cite{Christiansen:2006} and see comments in what follows). One of our
contributions here is to provide a robust method for constructing such
examples which relies only on the existence of a \textquotedblleft
good\textquotedblright\ Poisson formula for resonances and elementary
propagation estimates on the wave operator which hold for any Riemannian manifold.

We now describe the geometric setting for our results in greater detail. Let
$\overline{X}$ \ be a compact manifold with boundary having dimension $m=n+1$,
and denote by $X$ the interior of $\overline{X}$. Suppose that $x$ is a
defining function for the boundary of $\overline{X}$, that is, a smooth
function on $\overline{X}$ with $x>0$ in $X$ which vanishes to first order on
$M=\partial\overline{X}$. \ Two such defining functions differ at most by a
smooth positive function that does not vanish at $\partial\overline{X}$. A
complete metric $g$ on $X$ with the property that $x^{2}g$ extends to a smooth
metric on $\overline{X}$ is called \emph{conformally compact}. As $x$ ranges
over admissible defining functions, the metrics
\[
h_{0}=\left.  x^{2}g\right\vert _{T^{\ast}\partial\overline{X}}%
\]
give $M$ a natural conformal structure. If $\left[  h_{0}\right]  $ denotes
the conformal class of $h_{0}$, the conformal manifold $(M,\left[
h_{0}\right]  )$ is called the \emph{conformal infinity} of $(X,g)$. A
motivating example is the case where $X$ is the quotient of real hyperbolic
$(n+1)$-dimensional space by a convex co-compact discrete group of isometries,
so that $X$ has infinite metric volume and no cuspidal ends.

A conformally compact manifold $(X,g)$ is called \emph{asymptotically
hyperbolic} if the sectional curvatures approach $-1$ as $x\downarrow0$, and
\emph{hyperbolic near infinity} if the sectional curvatures of $g$ are
identically $-1$ outside a compact subset $K$ of $X$. Finally, $(X,g)$ is
\emph{strongly hyperbolic near infinity} \ if the following slightly more
stringent condition holds: there is a compact subset $K$ of $X$, a convex
co-compact hyperbolic manifold $(X_{0},g_{0})$, and a compact subset $K_{0}$
of $X_{0}$ so that $(X-K,g)$ is isometric to $(X_{0}-K_{0},g_{0})$. We will
consider scattering theory and resonances for manifolds hyperbolic near infinity.

We recall some fundamental results in the spectral and scattering theory for
asymptotically hyperbolic manifolds. See the papers of Mazzeo-Melrose
\cite{MM:1987}, Joshi-Sa Barreto \cite{JsB:2000,JsB:2001} for spectral and
scattering on asymptotically hyperbolic manifolds, see the papers of
Guillop\'{e}-Zworski \cite{GZ:1995,GZ:1997,GZ:1999} for spectral and
scattering on manifolds hyperbolic near infinity, and see the papers of
Graham-Zworski \cite{GZ} and Guillarmou
\cite{Guillarmou:2005,Guillarmou:2005c} for further results on scattering
resonances and resolvent resonances. A survey and further references can be
found in \cite{Perry:2007}.

If $\left(  X,g\right)  $ is hyperbolic near infinity, the positive Laplacian
$\Delta_{g}$ on $X$ has at most finitely many discrete eigenvalues and
continuous spectrum in $[n^{2}/4,\infty)$. The resolvent%
\begin{equation}
R_{g}(s)=\left(  \Delta_{g}-s(n-s)\right)  ^{-1}, \label{eq:resolvent1}%
\end{equation}
initially defined for $\Re(s)>n/2$, extends to a meromorphic family of
operators mapping $\mathcal{C}_{0}^{\infty}(X)$ into $\mathcal{C}^{\infty}%
(X)$. The singularities of the meromorphically continued resolvent (excepting
essential singularities)\ are called \emph{resolvent resonances}. At each
resolvent resonance $\zeta$ the resolvent has a Laurent series with finite
polar part whose coefficients are finite-rank operators. If $\left(
X,g\right)  $ is hyperbolic near infinity, the resolvent has no essential
singularities, as the construction in \cite{GZ:1995} shows. The multiplicity
of a resolvent resonance $\zeta$ is given by%
\begin{equation}
m_{g}(\zeta)=\operatorname*{rank}~\operatorname*{Res}_{s=\zeta}~R_{g}(s).
\label{eq:multiplicity1}%
\end{equation}
Note that there may be finitely many poles $\zeta$ with $\Re(\zeta)>n/2$
corresponding to the finitely many eigenvalues $\lambda=\zeta(n-\zeta)$ of
$\Delta_{g}$. We denote by $\mathcal{R}_{g}$ the resolvent resonances of
$\Delta_{g}$, counted with multiplicity.

Our interest lies in the asymptotic behavior of the counting function for
resolvent resonances:%
\begin{equation}
N_{g}(r)=\#\left\{  \zeta\in\mathcal{R}_{g}:\left\vert \zeta-n/2\right\vert
\leq r\right\}  . \label{eq:countingfunction1}%
\end{equation}
Optimal upper bounds of the form $N_{g}(r)\leq Cr^{n+1}$ were proven by
Cuevas-Vodev \cite{CV:2003} and Borthwick \cite{Borthwick:2008}, but, for
reasons that we will explain, optimal lower bounds for resolvent resonances
are more difficult to obtain. In the case $n=1$, Guillop\'{e} and Zworski
proved sharp upper \cite{GZ:1997} and lower \cite{GZ:1999} bounds.

We will study the distribution of resolvent resonances using the Poisson
formula for resonances obtained by Guillop\'{e} and Zworski for $n=1$ in
\cite{GZ:1997} and in the present setting by the first author in
\cite{Borthwick:2008}. To state it, we recall the $0$-trace, a regularization
introduced by Guillop\'{e} and Zworski \cite{GZ:1997} and inspired by the
$b$-integral of Melrose \cite{Melrose:1993}. First, the $0$-integral of a
function $f\in\mathcal{C}^{\infty}(X)$, polyhomogeneous in $x$ as
$x\downarrow0$, is defined to be
\[
\int^{0}~f~dg=\operatorname*{FP}_{\varepsilon\downarrow0}\int_{x>\varepsilon
}f~dg,
\]
and for an operator $A$ with smooth kernel we define the $0$-trace to be the
$0$-integral of the kernel of $A$ on the diagonal. The $0$-volume of $\left(
X,g\right)  $, denoted $%
\operatorname{0-Vol}%
(X,g)$, is simply $\int^{0}~dg$ and is known to be independent of the choice
of $x$ if the dimension of $X$ is even. In \cite{Borthwick:2008}, Borthwick
proved that if $(X,g)$ is strongly hyperbolic near infinity, then%
\begin{equation}%
\operatorname{0-Tr}%
\cos(t\sqrt{\Delta_{g}-n^{2}/4})=\sum_{\zeta\in\mathcal{R}_{g}^{\mathrm{sc}}%
}e^{(\zeta-n/2)\left\vert t\right\vert }-A(X)\frac{\cosh(\left\vert
t\right\vert /2)}{2\left(  \sinh\left(  \left\vert t\right\vert /2\right)
\right)  ^{n+1}} \label{eq:wave0trace1}%
\end{equation}
where
\begin{equation}
A(X)=\left\{
\begin{array}
[c]{ccl}%
0, &  & n\text{ odd,}\\
&  & \\
\left\vert \chi(\overline{X})\right\vert , &  & n\text{ even}%
\end{array}
\right.  \label{eq:AX}%
\end{equation}
and the left-hand side is a distribution on $\mathbb{R}\backslash\left\{
0\right\}  $, where $\chi(\overline{X})$ is the Euler characteristic of
$\overline{X}$ viewed as a compact manifold with boundary. The set
$\mathcal{R}_{g}^{\mathrm{sc}}$ is the set of \emph{scattering resonances} of
$\Delta_{g}$, a set which contains the resolvent resonances but also contains
new singularities which arise owing to the conformal infinity. The scattering
resonances are singularities of the scattering operator for $\Delta_{g}$,
which we now describe.

Fix a defining function $x$ for $\partial\overline{X}$ and consider the
Dirichlet problem for given $s\in\mathbb{C}$ and $f\in\mathcal{C}^{\infty}%
(M)$:%
\begin{align}
\left(  \Delta_{g}-s(n-s)\right)  u  &  =0\label{eq:dirichlet1}\\
u  &  =x^{n-s}F+x^{s}G\nonumber\\
\left.  F\right\vert _{\partial\overline{X}}  &  =f.\nonumber
\end{align}
Here, the functions $F$ and $G$ are restrictions to $X$ of smooth functions on
$\overline{X}$. The\ Dirichlet problem (\ref{eq:dirichlet1}) has a unique
solution if $\Re(s)=n/2$, $s\neq n/2$, so that for such $s$ the map%
\begin{align}
S_{g}(s)  &  :\mathcal{C}^{\infty}\left(  \partial\overline{X}\right)
\rightarrow\mathcal{C}^{\infty}(\partial\overline{X})\label{eq:s-matrix1}\\
f  &  \mapsto\left.  G\right\vert _{\partial\overline{X}}\nonumber
\end{align}
is well-defined and unitary. The scattering operator extends to a meromorphic
operator-valued function of $s$, but with poles whose residues have infinite
rank. If we renormalize and set%
\begin{equation}
\widetilde{S}_{g}(s)=\frac{\Gamma(s-n/2)}{\Gamma(n/2-s)}S_{g}(s),
\label{eq:s-matrix2}%
\end{equation}
the poles with infinite-rank residues are removed and all poles of
$\widetilde{S}_{g}(s)$ have finite-rank residues. Poles of $\widetilde{S}%
_{g}(s)$ are called \emph{scattering resonances}, and the multiplicity of a
scattering resonance $\zeta$ is given by
\begin{equation}
\nu_{g}(\zeta)=-\operatorname*{tr}~\operatorname*{Res}_{s=\zeta}\left[
\widetilde{S}_{g}^{\prime}(s)\widetilde{S}_{g}(n-s)\right]  .
\label{eq:multiplicity2}%
\end{equation}
We denote by $\mathcal{R}_{g}^{\mathrm{sc}}$ the set of scattering resonances
for $\Delta_{g}$, counted with multiplicity and we denote by $N_{g}%
^{\mathrm{sc}}(r)$ the counting function analogous to
(\ref{eq:countingfunction1}):%
\[
N_{g}^{\mathrm{sc}}(r)=\#\left\{  \zeta\in\mathcal{R}_{g}^{\mathrm{sc}%
}:\left\vert \zeta-n/2\right\vert \leq r\right\}  .
\]
It is the multiplicities of the scattering resonances that enter into the
Poisson formula (\ref{eq:wave0trace1}).

If $(X,g)$ is strongly hyperbolic near infinity, it is shown in
\cite{Borthwick:2008} that the following lower bounds, which take different
forms depending on whether $\dim(X)$ is even or odd, hold. If $\dim(X)$ is
even (i.e., \thinspace$n$ is odd), one has
\begin{equation}
N_{g}^{\mathrm{sc}}(r)\geq c\left\vert
\operatorname{0-Vol}%
(X,g)\right\vert r^{n+1} \label{eq:lb-n-odd}%
\end{equation}
for some $c>0$ and $r$ large (this result was already proved by
Guillop\'{e}-Zworksi in case $n=1$, where $N_{g}(r)=N_{g}^{\mathrm{sc}}(r)$).
On the other hand, if $\dim(X)$ is odd (i.e., $n$ is even), the lower bound
takes the form%
\begin{equation}
N_{g}^{\mathrm{sc}}(r)\geq c\left\vert \chi(\overline{X})\right\vert r^{n+1}
\label{eq:lb-n-even}%
\end{equation}
where $c>0$, $r$ is sufficiently large. Although we consider the more general
case of manifolds hyperbolic near infinity (i.e., dropping the
\textquotedblleft strongly\textquotedblright), this dichotomy will play an
important role in our work.

The scattering resonances include both resolvent resonances and an additional
set of singularities related to the conformal infinity. These singularities
occur at $s=n/2+k$ for $k=1,2,\cdots$; at these points, the residue of the
scattering operator $S_{g}(s)$ is an elliptic operator $P_{k}$ on $M$ with
kernel having finite dimension $d_{k}$. The operators $P_{k}$ are the GJMS
operators \cite{GJMS} associated to the conformal infinity: their connection
to scattering theory was elucidated by Graham and Zworski \cite{GZ}.

The precise relation between the respective multiplicities
(\ref{eq:multiplicity1}) and (\ref{eq:multiplicity2}) for resolvent resonances
and scattering resonances was partially established Guillop\'{e}-Zworski
($n=1$) and Borthwick-Perry ($n\geq1$) \cite{BP:2002}, and completed by
Guillarmou \cite{Guillarmou:2005}):%
\begin{equation}
\nu_{g}(\zeta)=m_{g}(\zeta)-m_{g}(n-\zeta)+\sum_{k\in\mathbb{N}}\left(
\mathbf{1}_{n/2-k}(\zeta)-\mathbf{1}_{n/2+k}(\zeta)\right)  d_{k}.
\label{eq:multiplicity3}%
\end{equation}
Here $\mathbf{1}_{t}(s)=1$ when $s=t$ and is zero elsewhere. This shows that
the difference between the counting functions for resolvent resonances and the
counting function for scattering resonances comes from two sources: first, the
finitely many $\zeta$ for which $n-\zeta$ corresponds to an eigenvalue of
$\Delta_{g}$ and second, the numbers $d_{k}$. If we let $\mathcal{R}%
_{g}^{\mathrm{GZ}}$ be the set $\left\{  n/2-k:k\in\mathbb{N}\right\}  $
assigning multiplicity $d_{k}$ to $\zeta=n/2-k$, and
\[
N_{g}^{\mathrm{GZ}}(r)=\#\left\{  \zeta\in\mathcal{R}_{g}^{\mathrm{GZ}%
}:\left\vert \zeta-n/2\right\vert \leq r\right\}  ,
\]
we have $N_{g}^{\mathrm{sc}}(r)=N_{g}^{\mathrm{GZ}}(r)+N_{g}(r)$ up to a
finite error which does not affect upper and lower bounds for large $r$ (this
was first pointed out in the literature by Guillarmou and Naud \cite{GN:2006}%
). Thus, in general, $N_{g}(r)\leq N_{g}^{\mathrm{sc}}(r)$, so that lower
bounds on $N_{g}^{\mathrm{sc}}(r)$ do not imply lower bounds on $N_{g}(r)$.

On the one hand, it is reasonable to expect that the counting function
$N_{g}(r)$, which is arguably a more natural counting function, obeys similar
bounds. On the other hand, there are known examples where $N_{g}^{\mathrm{GZ}%
}(r)$ saturates the lower bound (see also the remarks following Theorem 1.3 in
\cite{Borthwick:2008}); indeed, if $X=\mathbb{H}^{n+1}$, real hyperbolic
$(n+1)$-dimensional space, and $n$ is even, then $N_{g}(r)=0$! (see
Guillarmou-Naud \cite{GN:2006} for further discussion). For this reason, one
can only expect optimal lower bounds to hold in a \textquotedblleft
generic\textquotedblright\ sense.

We will say that the counting function $N_{g}(r)$ has \emph{maximal order of
growth} if $\rho=n+1$, in correspondence to the known upper bounds. If
$N_{g}(r)$ does not have maximal order of growth we will say that $g$ is
\emph{resonance-deficient}. Our main result says that the counting function
$N_{g}(r)$ has maximal order of growth for generic metrics in the following
sense. Let us fix a manifold $\left(  X,g_{0}\right)  \,$, assumed hyperbolic
near infinity, and a compact subset $K$ of $X$. Let $\mathcal{G}(g_{0},K)$ be
the set of metrics $g$ with $g=g_{0}$ outside $K$, and let $\mathcal{M(}%
g_{0},K)$ be the subset of $\mathcal{G}(g_{0},K)$ consisting of metrics for
which $N_{g}(r)$ has maximal order of growth. Viewing metrics as sections of
$\mathcal{C}^{\infty}(T^{\ast}X\otimes T^{\ast}X)$, we topologize these sets
with the $\mathcal{C}^{\infty}$ topology. This topology is compatible with
norm resolvent convergence for the corresponding Laplacians.

\begin{theorem}
\label{thm:main}Suppose that $\left(  X,g_{0}\right)  $ is hyperbolic near
infinity, and $K$ is a compact subset of $X$. Then:\newline(i) If $n$ is odd,
$\mathcal{M}(g_{0},K)$ contains an open dense subset of $\mathcal{G}(g_{0}%
,K)$.\newline(ii) If $n$ is even, $\mathcal{M}(g_{0},K)$ is a dense
$\mathcal{G}_{\delta}$ set in $\mathcal{G}(g_{0},K)$.
\end{theorem}

\begin{remark}
If $n=1$, it is known that $N_{g}^{\mathrm{GZ}}(r)=0$ so that $N_{g}%
^{\mathrm{sc}}(r)=N_{g}(r)$ and $\mathcal{M}(g_{0},K)=\mathcal{G}(g_{0},K)$
for any $K\subset X$; see \cite{GZ:1997} and \cite[section 8.5]{B:2007}.
\end{remark}

\begin{remark}
Theorem \ref{thm:main} gives a precise meaning to our assertion that optimal
lower bounds hold for ``generic'' metrics.
\end{remark}

\begin{remark}
For $n$ odd, we actually prove a stronger statement, that resonance-deficient
metrics can occur for at most one value of the zero-volume.
\end{remark}

A key observation is that compact metric perturbations leave $N_{g}%
^{\mathrm{GZ}}(r)$ unchanged since these resonances depend only on the
conformal infinity of $(X,g)$; thus it is natural to study the relative wave
trace for the perturbed and unperturbed metrics.

The contents of this paper are as follows. In section \ref{sec:complexmetric1}%
, we consider a family of complexified metrics
\[
g_{z}=(1-z)g_{0}+zg_{1}%
\]
for $z$ in a small complex neighborhood of $\left[  0,1\right]  $. Since this
is not a family of Riemannian metrics, we study the analog of the Laplacian
for $g_{z}$ and its scattering operator$\frac{{}}{{}}$. We then consider the
relative wave trace between $g_{0}$ and a compactly supported perturbation
$g_{1}$ in section \ref{sec:relwavetrace1}, and prove the first part of
Theorem \ref{thm:main}. Next, in section \ref{sec:tumor}, we construct a
compactly supported metric perturbation $g_{1}$ of $g_{0}$ obeying the optimal
lower bound. Finally, in \ref{sec:generic1}, we extend the methods of
\cite{Christiansen:2007} to prove the second part of Theorem \ref{thm:main}.

\section{Interpolated Laplacian and relative scattering matrix}

\label{sec:complexmetric1}

Let $(X,g_{0})$ be conformally compact and hyperbolic near infinity, and
$g_{1}$ another metric on $X$ that agrees with $g_{0}$ outside some compact
set $K\subset X$. For $z$ in the rectangular region,
\begin{equation}
\Omega_{\varepsilon}:=[-\varepsilon,1+\varepsilon]\times i[-\varepsilon
,\varepsilon], \label{eq:compldomain1}%
\end{equation}
we define a bilinear form interpolating between the two metrics by
\begin{equation}
g_{z}=(1-z)g_{0}+zg_{1}. \label{eq:complex1}%
\end{equation}
Let $P_{g_{z}}$ be the \textquotedblleft Laplacian\textquotedblright%
\ associated to $g_{z}$ in the formal sense,
\[
P_{g_{z}}:=-\frac{1}{\sqrt{\det g_{z}}}\partial_{j}[\sqrt{\det g_{z}}%
(g_{z})^{jk}]\>\partial_{k}.
\]
Assuming that $\varepsilon$ is sufficiently small, $\det g_{z}$ will lie
within the natural branch of the square root, and the coefficients of
$P_{g_{z}}$ will be analytic in $z$. With $z=a+ib$ for $a,b\in\mathbb{R}$, we
regard $P_{g_{z}}$ as an unbounded operator on $L^{2}(X,dg_{a})$.

The goal of this section is to define an operator $S_{g_{z}}(s)$ as the
scattering matrix associated to $P_{g_{z}}$. Since $P_{g_{z}}$ is not
self-adjoint, various facts need to be checked.


\subsection{Analytic continuation of the resolvent of $P_{g_{z}}$}

\label{subsec:analycont1}

We first prove that the resolvent of $P_{g_{z}}$, written as $(P_{g_{z}} -
s(n-s))^{-1}$, admits an analytic continuation in $s$.

\begin{lemma}
\label{presolvent.est} Assuming $\varepsilon$ is sufficiently small, there
exist $a_{\varepsilon}, C_{\varepsilon}$ independent of $z$, such that for
$\Re s > a_{\varepsilon}\geq n$, the operator $P_{g_{z}} - s(n-s)$ is
invertible and the inverse satisfies
\[
\left\Vert (P_{g_{z}} - s(n-s))^{-1} \right\Vert _{L^{2}(X, dg_{a})} \le
\frac{C_{\varepsilon}}{\Re(s)}.
\]

\end{lemma}

\begin{proof}
Since $P_{g_{a}}=\Delta_{g_{a}}$, the Laplacian of an actual metric $g_{a}$,
$R_{g_{a}}(s)$ is analytic for $\Re(s)>n$. Consider the simple identity
\begin{equation}
(P_{g_{z}}-s(n-s))R_{g_{a}}(s)=I+(P_{g_{z}}-P_{g_{a}})R_{g_{a}}(s).
\label{pzpa}%
\end{equation}
Since $P_{g_{z}}-P_{g_{a}}$ is a compactly supported second order differential
operator and $R_{g_{a}}(s)$ has order $-2$, the operator norm of $(P_{g_{z}%
}-P_{g_{a}})R_{g_{a}}(s)$ may be estimated for all $\Re(s)$ sufficiently large
by the supremum of the coefficients of $P_{g_{z}}-P_{g_{a}}$. These
coefficients are clearly $O(\varepsilon)$, so by choosing $\varepsilon$ small
we may assume
\[
\left\Vert (P_{g_{z}}-P_{g_{a}})R_{g_{a}}(s)\right\Vert \leq\tfrac{1}{2}\text{
for all }\Re(s)>a_{\varepsilon}.
\]
This shows that the right side of (\ref{pzpa}) is invertible, and hence that
$P_{g_{z}}-s(n-s)$ is invertible. The norm estimate on the inverse then
follows immediately from the Neumann series estimate,
\[%
\begin{split}
\left\Vert \lbrack I+(P_{g_{z}}-P_{g_{a}})R_{g_{a}}(s)]^{-1}\right\Vert  &
\leq\sum_{l=0}^{\infty}\left\Vert (P_{g_{z}}-P_{g_{a}})R_{g_{a}}(s)\right\Vert
^{l}\\
&  \leq2\quad\text{for }\Re(s)>a_{\varepsilon},
\end{split}
\]
and the standard resolvent estimate on $R_{g_{a}}(s)$, which for $\Re(s)\geq
n$ gives
\[
\left\Vert R_{g_{a}}(s)\right\Vert \leq\frac{1}{|s(n-s)|}.
\]

\end{proof}

Since $P_{g_{z}}$ agrees with $\Delta_{g_{0}}$ outside $K$, Lemma
\ref{presolvent.est} leads almost immediately to a proof of analytic
continuation of the resolvent of $P_{g_{z}}$. Recall that $x$ is a boundary
defining function for the boundary $\partial\overline{X}$, and let
$\mathcal{B}_{N}$ denote the bounded operators from $x^{N} L^{2}(X, dg_{a})
\to x^{-N} L^{2}(X, dg_{a})$.

\begin{proposition}
\label{pcontinue} The resolvent $R_{g_{z}}(s):=(P_{g_{z}}-s(n-s))^{-1}$, which
by Lemma \ref{presolvent.est} is defined for $z\in\Omega_{\varepsilon}$ and
$\Re(s)>a_{\varepsilon}$, admits for any $N>0$ a finitely meromorphic
continuation as a $\mathcal{B}_{N}$-valued function of $s$ to the region
$\Re(s)>-N+\tfrac{n}{2}$. For $(z,s)\in\Omega_{\varepsilon}\times(\Re
(s)>n/2)$, $R_{z}(s)$ is meromorphic in two variables as a $\mathcal{B}_{N}$
operator-valued function.
\end{proposition}

\begin{proof}
The resolvent $R_{g_{a}}(s)$ serves as a suitable parametrix for $R_{g_{z}%
}(s)$ near the boundary. Let $\chi, \chi_{0}, \chi_{1} \in C^{\infty}(X)$ be
cutoff functions vanishing in some neighborhood of $K$ and equal to 1 in some
neighborhood of ${\partial\bar X}$, such that $\chi=1$ on the support of
$\chi_{0}$ and $\chi_{1}=1$ on the support of $\chi$. Then for large $s_{0}>0$
we set
\begin{equation}
\label{Rparametrix}M(s) = (1-\chi_{0}) R_{g_{z}}(s_{0}) (1-\chi) + \chi_{1}
R_{g_{a}}(s) \chi.
\end{equation}
Then, using the facts that $\chi_{1} \chi= \chi$ and $(1- \chi) (1- \chi_{0})
= (1- \chi)$, we obtain
\begin{equation}
\label{pmik}(P_{g_{z}} - s(n-s)) M(s) = I - K_{1}(s) - K_{2}(s),
\end{equation}
where
\[
K_{1}(s) := [\Delta_{g_{0}}, \chi_{0}] R_{g_{z}}(s_{0}) (1-\chi) +
(s_{0}(n-s_{0}) - s(n-s)) (1-\chi_{0}) R_{g_{z}}(s_{0}) (1-\chi)
\]
and
\[
K_{2}(s) := [\Delta_{g_{0}}, \chi_{1}] R_{g_{a}}(s) \chi.
\]

The error term $K_{1}(s)$ is a compactly supported pseudodifferential operator
of order $-2$, whose operator norm may be made arbitrarily small by choosing
$s_{0}$ large, according to Lemma \ref{presolvent.est}. The error term
$K_{2}(s)$ has a smooth kernel contained in $x^{\infty}{x^{\prime}}^{s}
C^{\infty}(X\times X)$. For $N>0$, $K_{2}(s)$ is a compact operator on $x^{N}
L^{2}(X, dg_{a})$ for $\Re s > -N + \tfrac{n}2$. Its norm may be made
arbitrarily small by choosing $\operatorname{Re} (s)$ large using to the
standard resolvent estimate on $R_{g_{a}}(s)$.

Since $K_{1}(s)$ and $K_{2}(s)$ are meromorphic both in $z$ and in $s$, the
analytic Fredholm theorem thus applies to show that $I-K_{1}(s)-K_{2}(s)$ is
invertible meromorphically on $\rho^{N}L^{2}(X,dg_{a})$ for $z\in
\Omega_{\varepsilon}$ and $\Re(s)>-N+\tfrac{n}{2}$.
\end{proof}


\subsection{Upper bounds on the resonance counting function for $P_{g_{z}}$}

\label{subsec:counting1}

Proposition \ref{pcontinue} allows us to define $\mathcal{R}_{g_{z}}$ as the
set of resonances $\zeta$ of $R_{g_{z}}(s)$, with multiplicities counted by
\[
m_{z}(\zeta) := \operatorname{rank} \operatorname{Res}_{\zeta}R_{g_{z}}(s).
\]
The associated resonance counting function is
\[
N_{g_{z}}(r) := \#\{\zeta\in\mathcal{R}_{g_{z}}:\> |\zeta| \le r\}.
\]
For real $z$, polynomial bounds on the growth of $N_{g_{z}}(r)$ were proven in
\cite{GZ:1995}, and an optimal upper bound on the growth of $N_{g_{z}}(r)$ was
proven by Cuevas-Vodev \cite{CV:2003} and Borthwick \cite{Borthwick:2008}. We
need to extend this bound to $z\in\Omega_{\varepsilon}$.

\begin{proposition}
\label{upper.bound} For $\varepsilon>0$ sufficiently small, there exists
$C_{\varepsilon}$ independent of $z \in\Omega_{\varepsilon}$ such that
\[
N_{g_{z}}(r) \le C_{\varepsilon}r^{n+1}.
\]

\end{proposition}

\begin{proof}
In the proofs cited above, the interior metric enters only in the interior
parametrix term, i.e., the first term on the right in (\ref{Rparametrix}).
Most of the work goes into estimation of the boundary terms, and these results
apply immediately to $P_{g_{z}}$ because $P_{g_{z}} = \Delta_{g_{0}}$ on $X-
K$.

In the argument from Cuevas-Vodev, the only estimate required of the interior
term is \cite[eq.~(2.24)]{CV:2003}, an estimate on the singular values the
operator $K_{1}(s)$ defined above. These estimates depend only on the fact
that $K_{1}(s)$ is compactly supported and of order $-2$. For $\varepsilon$
sufficiently small, $P_{g_{z}}$ will be uniformly elliptic for $z\in
\Omega_{\varepsilon}$, and so $R_{g_{z}}(s)$ will have order $-2$ and the
required estimates on $K_{1}(s)$ can be done uniformly in $z$. The proof of
\cite[Prop.~1.2]{CV:2003} then gives a bound
\[
\#\{\zeta\in\mathcal{R}_{g_{z}}:\> |\zeta| \le r,\> \arg(\zeta- \tfrac{n}2)
\in[-\pi+\varepsilon, \pi- \varepsilon]\} \le C_{\varepsilon}r^{n+1}.
\]

To fill in the missing sector containing the negative real axis, we apply the
argument from Borthwick \cite{Borthwick:2008}. Here the interior parametrix
enters only in the proof of \cite[Lemma~5.2]{Borthwick:2008}. In the original
version, the standard resolvent estimate was used in the form $\left\Vert
R_{g_{a}}(n-s) \right\Vert = O(1)$ for $\Re s \le0$. For $R_{g_{z}}(n-s)$ this
must be replaced by the estimate from Lemma \ref{presolvent.est}, which gives
$\left\Vert R_{g_{z}}(n-s) \right\Vert = O(1)$ for $\Re s < n - a_{\varepsilon
}$. The result is that we have
\[
\#\{\zeta\in\mathcal{R}_{g_{z}}:\> |\zeta| \le r,\> \arg(\zeta-
n+a_{\varepsilon}) \in[\tfrac{\pi}2+\varepsilon, \tfrac{3\pi}2 -
\varepsilon]\} \le C_{\varepsilon}r^{n+1}.
\]

Since the two estimates obtained cover all but a compact region, the result follows.
\end{proof}


\subsection{The scattering matrix associated with $P_{g_{z}}$}

\label{subsec:s-matrix1}

The meromorphic continuation of $R_{g_{z}}(s)$ allows us to define the
associated scattering matrix $S_{g_{z}}(s)$ exactly as in (\ref{eq:dirichlet1}%
)-(\ref{eq:s-matrix1}). Scattering multiplicities are defined by
\[
\nu_{g_{z}}(\zeta):=-\operatorname{tr}\bigl[\operatorname{Res}_{\zeta}%
{\tilde{S}}_{g_{z}}^{\prime}(s){\tilde{S}}_{g_{z}}(s)^{-1}\bigr],
\]
where
\[
{\tilde{S}}_{g_{z}}(s):=\frac{\Gamma(s-\frac{n}{2})}{\Gamma(\frac{n}{2}%
-s)}S_{g_{z}}(s).
\]
Since the relation between scattering poles and resonances depends only on the
boundary structure of the resolvent, it carries over immediately to $S_{g_{z}%
}(s)$,
\begin{equation}
\nu_{g_{z}}(\zeta)=m_{z}(\zeta)-m_{z}(n-\zeta)+\sum_{k\in\mathbb{N}%
}\Bigl(\mathbf{1}_{n/2-k}(\zeta)-\mathbf{1}_{n/2+k}(\zeta)\Bigr)d_{k},
\label{nuz.muz}%
\end{equation}
where
\[
d_{k}=\dim\ker P_{k}%
\]
with
\[
P_{k}={\tilde{S}}_{g_{0}}(\tfrac{n}{2}+k).
\]

Applying $R_{g_{z}}(s)$ to (\ref{pmik}) from the left, we obtain the identity
\[
R_{g_{z}}(s)=M(s)+R_{g_{z}}(s)(K_{1}(s)+K_{2}(s))
\]
By taking the boundary limits of this formula as the boundary defining
functions $x,x^{\prime}\rightarrow0$, we obtain some useful relations. The
Poisson operators associated to $P_{g_{z}}$ and $\Delta_{g_{0}}$ are related
by
\begin{equation}
E_{g_{z}}(s)=E_{g_{0}}(s)+R_{g_{z}}(s)[\Delta_{g_{0}},\chi_{1}]E_{g_{0}}(s),
\label{Ez.E0}%
\end{equation}
and for the scattering matrices we have
\begin{equation}
S_{g_{z}}(s)=S_{g_{0}}(s)+E_{g_{z}}(s)^{t}[\Delta_{g_{0}},\chi_{1}]E_{g_{0}%
}(s). \label{Sz.S0}%
\end{equation}
The latter equation shows that $S_{g_{z}}(s)$ and $S_{g_{0}}(s)$ differ by a
smoothing operator on ${\partial\bar{X}}$. This shows in particular that the
relative scattering matrix $S_{g_{z}}(s)S_{g_{0}}(s)^{-1}$ is determinant
class. In fact, by the identity $E(s)S(s)^{-1}=-E(n-s)$, the relative
scattering matrix is given explicitly by
\begin{equation}
S_{g_{z}}(s)S_{g_{0}}(s)^{-1}=I-E_{g_{z}}(s)^{t}[\Delta_{g_{0}},\chi
_{1}]E_{g_{0}}(n-s) \label{SzS0inv}%
\end{equation}
We can exploit these relationships further by substituting the transpose of
(\ref{Ez.E0}) into (\ref{SzS0inv}). This yields
\begin{equation}%
\begin{split}
S_{g_{z}}(s)S_{g_{0}}(s)^{-1}  &  =I-E_{g_{0}}(s)^{t}[\Delta_{g_{0}},\chi
_{1}]E_{g_{0}}(n-s)\\
&  \qquad-([\Delta_{g_{0}},\chi_{1}]E_{g_{0}}(s))^{t}R_{g_{z}}(s)[\Delta
_{g_{0}},\chi_{1}]E_{g_{0}}(n-s).
\end{split}
\label{Srel.identity}%
\end{equation}
The point of this formula is that the dependence on $g_{z}$ is isolated in the
$R_{g_{z}}(s)$ term. It also shows that $S_{g_{z}}(s)S_{g_{0}}(s)^{-1}$ is a
meromorphic function of $z$ and $s$ since the same is true of $R_{g_{z}}%
(s)$.\ We will use it later to estimate $S_{g_{z}}(s)S_{g_{0}}(s)^{-1}$ in
terms of the difference in the metrics. Note that, since $R_{g_{z}}(s)$ is
meromorphic in $\Omega_{\varepsilon}\times\mathbb{C}$, so is $S_{g_{z}}(s)$.

Let $H_{z}(s)$ denote the Hadamard product over the resonance set
$\mathcal{R}_{g_{z}}$:
\begin{equation}
\label{pz.def}H_{z}(s) := \prod_{\zeta\in\mathcal{R}_{g_{z}}} E\Bigl(\frac
{s}{\zeta}, n+1\Bigr),
\end{equation}
where
\[
E(u,p) := (1-u) \exp\Bigl(u + \frac{u^{2}}2 + \dots+ \frac{u^{p}}p \Bigr).
\]
The relative scattering determinant may be defined as
\[
\sigma_{g_{z},g_{0}}(s):=\det[ S_{g_{z}}(s) S_{g_{0}}^{-1} (s)].
\]

\begin{proposition}
\label{detsrel.factor} The relative scattering determinant admits a
factorization
\begin{equation}
\label{detsrel.pp}\sigma_{g_{z},g_{0}}(s) = e^{q(s)} \frac{H_{z}(n-s)}%
{H_{z}(s)} \frac{H_{0}(s)}{H_{0}(n-s)},
\end{equation}
where $q(s)$ is a polynomial of degree at most $n+1$.
\end{proposition}

\begin{proof}
Let $A(s)$ be the auxiliary operator introduced in \cite[\S 3]{Borthwick:2008}%
, defined so that $S_{g_{z}}(s) - A(s)$ is smoothing. Note that the
construction of $A(s)$ depends only on the metric in a neighborhood of
${\partial\bar X}$ and so the same $A(s)$ works for any of the ``metrics''
$g_{z}$. We set
\[
\vartheta_{z}(s) := \det S_{g_{z}}(n-s) A(s).
\]
The arguments in \cite[\S 6]{Borthwick:2008} apply immediately to show that
$\vartheta_{z}(s)$ is a ratio of entire functions of bounded order.
Furthermore
\[
\det S_{g_{z}}(s) S_{g_{0}}(s)^{-1} = \frac{\vartheta_{0}(s)}{\vartheta
_{z}(s)}.
\]

In computing the divisor of $\vartheta_{0}(s)/\vartheta_{z}(s)$, the terms
coming from $A(s)$ cancel, and we find, by the definition of $\nu_{g_{z}%
}(\zeta)$,
\[
\operatorname{Res}_{\zeta}\frac{\vartheta_{z}^{\prime}}{\vartheta_{z}}(s) -
\operatorname{Res}_{\zeta}\frac{\vartheta_{0}^{\prime}}{\vartheta_{0}}(s) = -
\nu_{g_{z}}(\zeta) + \nu_{g_{0}}(\zeta).
\]
Hence the relation (\ref{nuz.muz}) shows that both sides of (\ref{detsrel.pp})
have the same divisor. We have thus proven (\ref{detsrel.pp}) with $q(s)$ some
polynomial of unknown degree.

To control the degree, we use Lemma~\ref{presolvent.est} to adapt the proof of
\cite[Lemma~5.2]{Borthwick:2008}, just as we did above, to prove for
$\Re(s)<n-a_{\varepsilon}$ that
\[
|\vartheta_{z}(s)|<e^{C_{\eta,\varepsilon}\langle s\rangle^{n+1}},
\]
provided $d(s,-\mathbb{N}_{0})>\eta$. Since we can write
\[
\vartheta_{z}(s)=e^{-q(s)}\frac{H_{0}(n-s)}{H_{0}(s)}\frac{H_{z}(s)}%
{H_{z}(n-s)}\>\vartheta_{0}(s),
\]
and the Hadamard products have order $n+1$, this shows that $|q(s)|\leq
C|s|^{n+1+\delta}$ in the half-plane $\Re(s)<n-a_{\varepsilon}$ for any
$\delta>0$. Hence the degree of $q(s)$ is at most $n+1$.
\end{proof}

Define the meromorphic function $\Upsilon_{z}(s)$ by
\[
\Upsilon_{z}(s)=(2s-n)%
\operatorname{0-Tr}%
[R_{g_{z}}(s)-R_{g_{z}}(1-s)],
\]
for $s\notin\mathbb{Z}/2$. The connection between $\Upsilon_{z}(s)$ and the
logarithmic derivative of the scattering determinant established by
Patterson-Perry \cite[Prop.~5.3 and Lemma~6.7]{PP:2001} depends only on the
structure of model neighborhoods near infinity, and so carries over to our
case without alteration. This yields the following Birman-Krein type formula:

\begin{proposition}
\label{birman.krein} For $s\notin\mathbb{Z}/2$ we have the meromorphic
identity,
\[
- \frac{d}{ds} \log\sigma_{g_{z},g_{0}}(s) = \Upsilon_{z}(s) - \Upsilon
_{0}(s).
\]

\end{proposition}

For $a$ real (so that $g_{a}$ is an actual metric), we define the relative
volume
\[
V_{\mathrm{rel}}(a)=\operatorname{Vol}(K,g_{a})-\operatorname{Vol}(K,g_{0}).
\]
We can derive asymptotics from Proposition 2.5 as in Borthwick
\cite[Thm.~10.1]{Borthwick:2008}. Furthermore, the restriction to metrics
strongly hyperbolic near infinity in \cite{Borthwick:2008} can be relaxed here
because we are only interested in the \emph{relative} scattering determinant.

\begin{corollary}
\label{weyl.asymp} For $a \in[-\varepsilon, 1+ \varepsilon]$, as $\xi
\to+\infty$,
\[
\log\sigma_{g_{a},g_{0}}(\tfrac{n}2 + i\xi) = c_{n} V_{\mathrm{rel}}(a)
\>\xi^{n+1} + \mathcal{O} (\xi^{n}).
\]
where
\[
c_{n} = -2\pi i \frac{(4\pi)^{-(n+1)/2}}{\Gamma(\frac{n+3}{2})}.
\]

\end{corollary}


\section{Lower bounds from the relative wave trace}

\label{sec:relwavetrace1}

If the dimension $n+1$ is even ($n$ odd), then we can deduce a lower bound on
the resolvent resonances by using a relative wave trace to cancel the
conformal Graham-Zworski scattering poles (the $d_{k}$ terms in Poisson
formula \cite[Thm.~1.2]{Borthwick:2008}).

Let $(X,g_{0})$ be conformally compact and hyperbolic near infinity, and
$g_{1}$ another metric that agrees with $g_{0}$ outside some compact set
$K\subset X$. By the functional calculus, $\Upsilon_{a}(\tfrac{n}{2}+i\xi)$ is
essentially the Fourier transform of the continuous part of the wave 0-trace
(see \cite[Lemma~8.1]{Borthwick:2008}). By Propositions~\ref{detsrel.factor}
and \ref{birman.krein} we can write
\[
\Upsilon_{1}(s)-\Upsilon_{0}(s)=\partial_{s}\log\left[  e^{q(s)}\frac
{H_{1}(s)}{H_{1}(n-s)}\frac{H_{0}(n-s)}{H_{0}(s)}\right]
\]
Taking the Fourier transform just as in the proof of \cite[Thm.~1.2]%
{Borthwick:2008} then gives

\begin{theorem}
\label{rel.poisson} For $(X,g_{0})$ conformally compact and hyperbolic near
infinity, and $g_{1}$ a compactly supported perturbation, we have
\[%
\begin{split}
&
\operatorname{0-Tr}%
\left[  \cos\left(  t\sqrt{\smash[b]{\Delta_{g_1} - n^2/4}}\,\right)  \right]
-%
\operatorname{0-Tr}%
\left[  \cos\left(  t\sqrt{\smash[b]{\Delta_{g_0} - n^2/4}}\,\right)  \right]
\\
&  \qquad=\frac{1}{2}\sum_{\zeta\in\mathcal{R}_{g_{1}}}e^{(\zeta
-n/2)|t|}-\frac{1}{2}\sum_{\zeta\in\mathcal{R}_{g_{0}}}e^{(\zeta-n/2)|t|},
\end{split}
\]
in the sense of distributions on $\mathbb{R}-\{0\}$.
\end{theorem}

(Note that \cite[Thm.~1.2]{Borthwick:2008} required a metric strongly
hyperbolic near infinity; we may drop that restriction here because we are
dealing with the difference of two wave traces.)

Theorem~\ref{rel.poisson} applies in any dimension, but it only gives a lower
bound on resonances when the singularity on the wave trace side spreads out
beyond $t=0$. The following Corollary requires $n+1$ even and a nonzero
relative volume between the two metrics.

\begin{corollary}
\label{cor:hominis}Assume that $n+1$ is even and $g_{0}$, $g_{1}$ are metrics
as above. There is a constant $c>0$ such that
\[
N_{g_{0}}(r)+N_{g_{1}}(r)\geq c\>\bigl|\operatorname{Vol}(K,g_{1}%
)-\operatorname{Vol}(K,g_{0})\bigr|\>r^{n+1}.
\]

\end{corollary}

\begin{proof}
For $\phi\in C_{0}^{\infty}(\mathbb{R}_{+})$ and $\lambda>0$ we can apply
\cite[Lemma~9.2]{Borthwick:2008} to obtain from Theorem~\ref{rel.poisson} the
asymptotic
\begin{align*}
\left\vert \sum_{\zeta\in\mathcal{R}_{g_{1}}}\widehat{\phi}(i(\zeta-\tfrac
{n}{2})/\lambda)-\sum_{\zeta\in\mathcal{R}_{g_{0}}}\widehat{\phi}%
(i(\zeta-\tfrac{n}{2})/\lambda)\right\vert  &  =c_{n}%
\>\bigl|\operatorname{Vol}(K,g_{1})-\operatorname{Vol}(K,g_{0})\bigr|\>\lambda
^{n+1}\\
&  +\mathcal{O}(\lambda^{n-1}),
\end{align*}
as $\lambda\rightarrow\infty$. Since $\phi$ is compactly supported, its
Fourier transform satisfies analytic estimates,
\[
|\hat{\phi}(\xi)|\leq C_{m}(1+|\xi|)^{-m},
\]
for $m\in\mathbb{N}$. Thus for $\lambda$ sufficiently large and setting
$m=n+2$,
\[%
\begin{split}
c_{n}\>\bigl|\operatorname{Vol}(K,g_{1})-\operatorname{Vol}(K,g_{0}%
)\bigr|\>\lambda^{n+1}  &  \leq\sum_{\zeta\in\mathcal{R}_{g_{0}}%
\cup\mathcal{R}_{g_{1}}}|\widehat{\phi}(i(\zeta-\tfrac{n}{2})/\lambda)|\\
&  \leq C\sum_{\zeta\in\mathcal{R}_{g_{0}}\cup\mathcal{R}_{g_{1}}}%
(1+|\zeta|/\lambda)^{-n-2},
\end{split}
\]
Then, if we let $M(r)=N_{g_{0}}(r)+N_{g_{1}}(r)$, we have
\[%
\begin{split}
c_{n}\>\bigl|\operatorname{Vol}(K,g_{1})-\operatorname{Vol}(K,g_{0}%
)\bigr|\>\lambda^{n+1}  &  \leq C\int_{0}^{\infty}(1+r/\lambda)^{-n-2}%
\>dM(r)\\
&  \leq C\int_{0}^{\infty}(1+r)^{-n-3}\>M(\lambda r)\>dr.
\end{split}
\]
Splitting the integral at $b$ and using the upper bound from
Proposition~\ref{upper.bound} to control the $[b,\infty)$ piece then yields
\[
c_{n}\>\bigl|\operatorname{Vol}(K,g_{1})-\operatorname{Vol}(K,g_{0}%
)\bigr|\>\lambda^{n+1}\leq CM(\lambda b)+C\lambda^{n+1}b^{-1}.
\]
Taking $b$ sufficiently large completes the proof.
\end{proof}

We conclude this section with:

\begin{proof}
[Proof of part (i) of Theorem \ref{thm:main}:]Suppose that $\dim(X)$ is even.
If $\mathcal{G}(g_{0},K)$ contains resonance-deficient metrics, then we may
redefine $g_{0}$ to assume that this background metric is resonance-deficient.
Observe that for a fixed compact subset $K$ of $X$, the function%
\begin{align*}
\mathcal{G}(g_{0},K)  &  \mapsto\mathbb{R}\\
g  &  \rightarrow%
\operatorname{0-Vol}%
(X,g)
\end{align*}
is continuous. Moreover, if we fix $g\in\mathcal{G}(g_{0},K)$ and $\varphi
\in\mathcal{C}_{0}^{\infty}(K)$, and consider the family%
\[
g_{t}=e^{t\varphi}g,
\]
we have
\[
\left.  \frac{d}{dt}\right\vert _{t=0}\left(
\operatorname{0-Vol}%
(X,g_{t})\right)  =\int\varphi~dg
\]
which is nonzero for any nonzero, nonnegative $\varphi\in\mathcal{C}%
_{0}^{\infty}(K)$.

By continuity,
\[
\mathcal{S}=\left\{  g\in\mathcal{G}(g_{0},K):%
\operatorname{0-Vol}%
(X,g)\neq%
\operatorname{0-Vol}%
(X,g_{0})\right\}
\]
is open in $\mathcal{G}(g_{0},K)$. By the conformal perturbation argument
above, $\mathcal{S}$ is also dense in $\mathcal{G}(g_{0},K)$. It follows from
Corollary \ref{cor:hominis} that $\mathcal{S}\subset\mathcal{M}(g_{0},K)$,
proving Theorem \ref{thm:main}(i).
\end{proof}

\section{A metric perturbation with optimal order of growth}

\label{sec:tumor}

In this section, we prove:

\begin{theorem}
\label{thm:tumor}Suppose that $\left(  X,g_{0}\right)  $ is hyperbolic near
infinity and $\dim(X)=n+1$. Suppose that $N_{g_{0}}(r)=o(r^{n+1})$ as
$r\rightarrow\infty$, let $x_{0}\in X$. There is a Riemannian metric $g_{1}$
on $X$ with the following properties: $g_{1}=g_{0}$ outside $B(x_{0},3)$, and
$N_{g_{1}}(r)\geq Cr^{n+1}$ for a strictly positive constant $C$ and
sufficiently large $r$.
\end{theorem}

The hypothesis of Theorem \ref{thm:tumor} implies that the distribution
$u_{0}(t)$ on $\mathbb{R}\backslash\left\{  0\right\}  $ defined by%
\begin{equation}
u_{0}(t)=\frac{1}{2}\sum_{\xi\in\mathcal{R}_{g_{0}}}e^{(\zeta-n/2)\left\vert
t\right\vert }, \label{eq:u0}%
\end{equation}
where $\mathcal{R}_{0}$ is the set of resolvent resonances for $g_{0}$,
satisfies%
\begin{equation}
\left\vert \widehat{\varphi u_{0}}(\lambda)\right\vert =o(\lambda^{n})
\label{eq:0small}%
\end{equation}
for any $\varphi\in\mathcal{C}_{0}^{\infty}(\mathbb{R}^{+})$. Let
\begin{equation}
u_{1}(t)=\sum_{\xi\in\mathcal{R}_{g_{1}}}e^{(\zeta-n/2)\left\vert t\right\vert
}. \label{eq:u1}%
\end{equation}
where $\mathcal{R}_{1}$ is is the set of resolvent resonances for $g_{1}$.
Following ideas of Sj\"{o}strand-Zworski \cite{SZ:1993}, we will construct a
perturbed metric which, geometrically, attaches a large sphere to $X$ at
$x_{0}$, and use wave trace estimates on $u_{1}-u_{0}$ and the following
Tauberian theorem \cite[p. 848]{SZ:1993} to prove a lower bound on the
counting function for the resonances of the perturbed metric.

\begin{theorem}
\label{thm:tauber} \cite{SZ:1993} Let $u_{1}\in\mathcal{D}^{\prime}%
(\mathbb{R})$ be the distribution associated with the resolvent resonance set
$\mathcal{R}_{g_{1}}$ as in (\ref{eq:u1}). Suppose that for some constants
$b,d>0$ and every $\varphi\in\mathcal{C}_{0}^{\infty}(\mathbb{R}^{+})$
supported in a sufficiently neighborhood of $d$ with $\varphi(d)=1$ and
$\widehat{\varphi}(\tau)\geq0,$ we have%
\[
\left\vert \widehat{\varphi u_{1}}(\lambda)\right\vert \geq(b-o(1))\lambda^{n}%
\]
as $\lambda\rightarrow+\infty$. Then, the resonance counting function
satisfies%
\[
N_{g_{1}}(r)\geq\left(  B-o(1)\right)  r^{n+1},~~B=b/(\pi(n+1)).
\]

\end{theorem}

Thus, we need to choose $g_{1}$ so that $\left\vert \widehat{\varphi u_{1}%
}(\lambda)\right\vert \geq C\lambda^{n}$ as $\lambda\rightarrow+\infty$. By
(\ref{eq:0small}) it suffices to prove the same estimate for $u_{1}-u_{0}$. It
follows from the relative Poisson formula, Theorem \ref{rel.poisson}, that
$u_{1}(t)-u_{0}(t)$ is a difference of wave traces.

Sj\"{o}strand and Zworski used this idea in the Euclidean setting to construct
scattering metrics which are Euclidean near infinity and whose resonance
counting function has optimal order of growth. In our setting, the background
metric is more complicated, so we begin with some perturbative estimates on
the wave trace.

Let $x_{0}\in X$ and denote by $B(x_{0},3)$ the ball of radius $3$ in the
\emph{unperturbed} metric. We consider metrics $g_{0}$ and $g_{1}$ on a
manifold $X$ so that $g_{1}=g_{0}$ on $X\backslash B(x_{0},3)$ and both
metrics are hyperbolic near infinity. We will make a specific choice of
$g_{1}$ later. We denote by $\Delta_{0}$ and $\Delta_{1}$ the respective
positive Laplace-Beltrami operators and set%
\[
Q_{0}=\left(
\begin{array}
[c]{cc}%
0 & I\\
-(\Delta_{0}-n^{2}/4) & 0
\end{array}
\right)  ,~~Q_{1}=\left(
\begin{array}
[c]{cc}%
0 & I\\
-(\Delta_{1}-n^{2}/4) & 0
\end{array}
\right)  ,
\]
where $n^{2}/4$ is the bottom of the continuous spectrum. These operators are
the infinitesimal generators of wave groups $U_{0}(t)=\exp(tQ_{0})$ and
$U_{1}(t)=\exp(tQ_{1})$ acting on the Hilbert spaces of initial data
$(v_{0},v_{1})$ of finite energy, defined as follows. Let $\left(  Y,g\right)
$ denote either $(X,g_{0})$ or $\left(  X,g_{1}\right)  $. Let $\mathcal{H}$
denote the completion of $\mathcal{C}_{0}^{\infty}(Y)\oplus\mathcal{C}%
_{0}^{\infty}(Y)$ in the norm%
\[
\left\Vert (v_{0},v_{1})\right\Vert _{Y}=\left\Vert \nabla v_{0}\right\Vert
+\left\Vert v_{1}\right\Vert
\]
where $\left\Vert ~\cdot~\right\Vert $ denotes the $L^{2}(Y,g)$ norm. Letting
$\dot{H}^{1}(Y,g)$ denote the completion of $\mathcal{C}_{0}^{\infty}(Y)$ in
the norm $\left\Vert \nabla(~\cdot~)\right\Vert $ modulo constants, we have
$\mathcal{H}=\dot{H}^{1}(Y,g)\oplus L^{2}(Y,g)$. An important remark (see, for
example, \cite[Chapter IV, Lemma 1.1]{LP:1989}) is that $\dot{H}%
^{1}(Y,g)\subset L_{\mathrm{loc}}^{2}(Y,g)$ and that the Sobolev bound%
\[
\left(  \int\left\vert v\right\vert ^{2(n+1)/(n-1)}dg\right)  ^{(n-1)/2(n+1)}%
\leq c\left(  \int\left\vert \nabla v\right\vert ^{2}\right)  ^{1/2}%
\]
holds (recall $\dim Y=n+1$). The wave groups $U_{0}(t)$ and $U_{1}(t)$ act as
unitary groups on their respective Hilbert spaces.

To make perturbative estimates, it is convenient to use the natural unitary
map $J:L^{2}(X,dg_{0})\rightarrow L^{2}(X,dg_{1})$ and define $U(t)=J^{\ast
}U_{1}(t)J$. The operators $U(t)$ are a unitary group on $\mathcal{H}_{0}$
with infinitesimal generator%
\[
Q=\left(
\begin{array}
[c]{cc}%
0 & I\\
-(\Delta-n^{2}/4) & 0
\end{array}
\right)
\]
where $\Delta$ is a second-order elliptic differential operator with
$\Delta=\Delta_{0}$ on functions with support contained in $X\backslash
B(x_{0},3)$.

We will be interested in Fourier transforms of the wave trace of the form
(\ref{eq:0small}) where $\varphi$ is localized near the period $T$ of a closed
geodesic. Let $\varphi\in\mathcal{C}_{0}^{\infty}([-1,1])$ with $\varphi
(0)=1/(2\pi)$ and $\widehat{\varphi}(\tau)\geq0$, and define%
\begin{equation}
\varphi_{\varepsilon,T}(t)=\varphi\left(  \frac{t-T}{\varepsilon}\right)  .
\label{eq:locfnc1}%
\end{equation}
Let $D(t)$ be the distribution%
\[
D(t)=%
\operatorname{0-Tr}%
(U(t)-U_{0}(t))
\]
and consider the Fourier transform
\begin{equation}
\Phi(\lambda)=\int e^{-i\lambda t}\varphi_{\varepsilon,T}(t)~D(t)~dt.
\label{eq:PhiLambda}%
\end{equation}
which is the difference of $\widehat{\varphi_{\varepsilon,T}u_{1}}$ and
$\widehat{\varphi_{\varepsilon,T}u_{0}}$. We will first isolate the dominant
term in $\Phi(\lambda)$ for a arbitrary compactly supported perturbation, and
then make a specific choice of $g_{1}$ that produces the desired
$\mathcal{O}(\lambda^{n})$ growth.

In what follows, it will be important to microlocalize in the unit cosphere
bundle $S^{\ast}X$. We denote by $\Pi:S^{\ast}X\rightarrow X$ the canonical
projection. For $(x,\xi)\in S^{\ast}X$, we denote by $\gamma_{t}(x,\xi)$ the
unit speed geodesic passing through $(x,\xi)$ at time zero. Unless otherwise
stated, the geodesics will be defined with respect to the perturbed metric on
$X$. Note that, on $X\backslash B(x_{0},3)$, these geodesics coincide with
those of $g_{0}$.

The first lemma allows us to localize the wave trace near the perturbation up
to controlled errors. Let $\psi\in\mathcal{C}_{0}^{\infty}(X)$ with
\begin{equation}
\psi(x)=\left\{
\begin{array}
[c]{ccc}%
1 &  & d(x_{0},x)<4\\
&  & \\
0 &  & d(x_{0,}x)>6
\end{array}
\right.  \label{eq:cutoff-psi}%
\end{equation}
where $d(~\cdot~,~\cdot~)$ is the distance in the unperturbed metric $g_{0}$.

\begin{lemma}
\label{lemma:wt1} The asymptotic formula%
\[
\Phi(\lambda)=\int e^{-i\lambda t}\varphi_{\varepsilon,T}%
(t)~\operatorname*{Tr}\left[  \left(  U(t)-U_{0}(t)\right)  \psi\right]
~dt+\mathcal{O}\left(  T\lambda^{n}\right)
\]
holds as $\lambda\rightarrow\infty$.
\end{lemma}

\begin{proof}
First, by finite propagation speed, it follows that $U(t)f=U_{0}(t)f$ for any
$t\in\operatorname*{supp}\varphi_{\varepsilon,T}$ and $f$ with support a
distance at least $2T$ from $B(x_{0},3)$. Hence, if
\[
\chi_{T}(x)=\left\{
\begin{array}
[c]{ccc}%
1 &  & d(x_{0},x)<2T\\
&  & \\
0 &  & d(x_{0,}x)>3T
\end{array}
\right.
\]
(where $d(~\cdot~,~\cdot~)$ is the distance in the unperturbed metric, and
$T>2$ say), we have%
\[
\Phi(\lambda)=\int e^{-i\lambda t}\varphi_{\varepsilon,T}%
(t)~\operatorname*{Tr}\left[  \left(  U(t)-U_{0}(t)\right)  \chi_{T}\right]
~dt.
\]
It suffices to show that
\begin{equation}
\int e^{-i\lambda t}\varphi_{\varepsilon,T}(t)~\operatorname*{Tr}\left[
\left(  U(t)-U_{0}(t)\right)  (1-\psi)\chi_{T}\right]  ~dt=\mathcal{O}%
(T\lambda^{n}) \label{eq:wt1}%
\end{equation}
since $\psi\chi_{T}=\psi$. Let $C\in\Psi_{\mathrm{phg}}^{0}(X)$ be a
pseudodifferential operator with the following properties:\footnote{See
Appendix \ref{app:wavetrace1} for the definition of the essential support $%
\operatorname{S-ES}%
$ of a pseudodifferential operator.}%
\begin{align}%
\operatorname{S-ES}%
(C)  &  \subset\left\{  (x,\xi)\in S^{\ast}X:\Pi\gamma_{t}(x,\xi)\in
B(x_{0},5),~\exists t\in\lbrack-1,4T]\right\}  ,\label{eq:ESC}\\
& \nonumber\\%
\operatorname{S-ES}%
(I-C)  &  \subset\left\{  (x,\xi)\in S^{\ast}X:\Pi\gamma_{t}(x,\xi)\notin
B(x_{0},4),~\forall t\in\lbrack-1,4T]\right\}  \label{eq:ESI-C}%
\end{align}
where $I$ denotes the identity operator, and the geodesics and balls are
understood to be defined with respect to $g_{0}$. We split%
\[
\left(  U(t)-U_{0}(t)\right)  (1-\psi)\chi_{T}=G_{1}(t)+G_{2}(t)
\]
where%
\begin{align*}
G_{1}(t)  &  =\left(  U(t)-U_{0}(t)\right)  \left(  I-C\right)  \left(
1-\psi\right)  \chi_{T},\\
G_{2}(t)  &  =\left(  U(t)-U_{0}(t)\right)  C\left(  1-\psi\right)  \chi_{T}.
\end{align*}

First, we claim that $G_{1}(t)$ is a smoothing operator for $t\in
\operatorname*{supp}\left(  \varphi_{\varepsilon,T}\right)  $. To see this,
note that $G_{1}(0)=0$ so by the Fundamental Theorem of Calculus%
\[
G_{1}(t)=\int_{0}^{t}U(t-s)(Q-Q_{0})U_{0}(s)\left(  I-C\right)  \chi
_{T}(1-\psi)~ds.
\]
Note that $Q-Q_{0}=0$ outside $B(x_{0},3)$, and let $\theta\in C^{\infty}(X)$
with
\[
\theta(x)=\left\{
\begin{array}
[c]{cl}%
1 & x\in B(x_{0},7/2),\\
& \\
0 & x\notin B(x_{0},15/4).
\end{array}
\right.
\]
(where again the balls are defined with respect to $g_{0}$). By the
propagation of singularities and (\ref{eq:ESI-C}), the operator $\theta
U_{0}(s)(I-C)$ has a smooth kernel for all $t\in\lbrack0,2T]$. Combining these
observations we see that%
\[
G_{1}(t)=\int_{0}^{t}U(t-s)(Q-Q_{0})\theta U_{0}(s)(I-C)\chi_{T}(1-\psi)~ds
\]
is a smoothing operator for $t\in\lbrack0,2T]$. It follows that
\begin{equation}
\int e^{-i\lambda t}\varphi_{\varepsilon,T}(t)~\operatorname*{Tr}%
G_{1}(t)~dt=\mathcal{O}\left(  \lambda^{-\infty}\right)  . \label{eq:wt11}%
\end{equation}

Next, we consider $G_{2}(t)$. The operator $C_{T}=C\left(  1-\psi\right)
\chi_{T}$ has $%
\operatorname{S-ES}%
(C_{T})$ contained in a subset of $S^{\ast}X$ having volume $\mathcal{O}(T)$
(compare Lemma \ref{lemma:phasespace1} below; here volume is unambiguously
given by $g_{0}$ since $\pi\left(
\operatorname{S-ES}%
(C_{T})\right)  $ lies away from the metric perturbation). We can then deduce
that
\begin{equation}
\int e^{-i\lambda t}\varphi_{\varepsilon,T}(t)~\operatorname*{Tr}%
G_{2}(t)~dt=\mathcal{O}\left(  T\lambda^{n}\right)  \label{eq:wt12}%
\end{equation}
by applying Lemma \ref{lemma.2} to the two respective terms involving $U(t)$
and $U_{0}(t)$. The estimate (\ref{eq:wt1}) follows from (\ref{eq:wt11}) and
(\ref{eq:wt12}).
\end{proof}

Next, we note:

\begin{lemma}
\label{lemma:wt2}The estimate%
\[
\int e^{-i\lambda t}\varphi_{\varepsilon,T}(t)~\operatorname*{Tr}\left[
U_{0}(t)\psi\right]  ~dt=\mathcal{O}_{\varepsilon,\psi}(\lambda^{n})
\]
holds.
\end{lemma}

\begin{proof}
An immediate consequence of Lemma \ref{lemma.2} with $B=\psi$.
\end{proof}

Combining Lemmas \ref{lemma:wt1} and \ref{lemma:wt2}, we have shown that
\begin{equation}
\Phi(\lambda)=\Phi_{1}(\lambda)+\mathcal{O}_{\varepsilon,\psi}(T\lambda^{n}).
\label{eq:wt-loc3}%
\end{equation}
where%
\[
\Phi_{1}(\lambda)=\int e^{-i\lambda t}\varphi_{\varepsilon,T}%
(t)~\operatorname*{Tr}\left[  U(t)\psi\right]  ~dt.
\]

We now make a choice of $g_{1}$ so that $(X,g_{1})$ is isometric to a manifold
$(X_{R},g_{R})$ defined as follows. Roughly, $X_{R}$ is $X$ with a ball
excised, and a large Euclidean sphere glued in analogy to the construction in
\cite{SZ:1993}. More precisely, denote by $\mathbb{S}^{m}(R)$ the Euclidean
sphere of radius $R$ and dimension $m$ with the usual metric. Pick a point
$x_{0}\in X$ and $x_{1}\in\mathbb{S}^{n+1}(R)$. The manifold $X_{R}$ consists
of $X\backslash B_{X}(x_{0},1)$ together with a cylindrical neck
$N=\mathbb{S}^{n}(1)\times\left[  0,1\right]  $ that connects $X\backslash
B(x_{0},1)$ to $\mathbb{S}^{n+1}(R)\backslash B_{\mathbb{S}^{n+1}(R)}%
(x_{1},1)$ (we make the natural identification between $\mathbb{S}^{n}(1)$ and
$\partial B(x_{0},1)\subset X$ on the one hand, and $\mathbb{S}^{n}(1)$ and
$\partial B(x_{1,}1)$ $\subset\mathbb{S}^{n+1}(R)$ on the other). Thus%
\[
X_{R}=\left(  X\backslash B_{X}(x_{0},1)\right)  \sqcup N\sqcup(\mathbb{S}%
^{n+1}(R)\backslash B_{\mathbb{S}^{n+1}(R)}(x_{1},1).
\]
We put a smooth metric $g_{R}$ on $X_{R}$ which coincides with the standard
metric on the sphere on $\mathbb{S}^{n+1}(R)\backslash B_{\mathbb{S}^{n+1}%
(R)}(x_{1},2)$, and the original metric $g_{0}$ on $X\backslash B(x_{0},3)$.
There is a natural diffeomorphism $f:X\rightarrow X_{R}$ and we take
$g_{1}=f^{\ast}g_{R}$.

\begin{figure}[ptb]
\psfrag{x0}{$x_0$} \psfrag{x1}{$x_1$} \psfrag{SR}{$\mathbb{S}^{n+1}(R)$}
\psfrag{X}{$X$} \psfrag{N}{$N$} \psfrag{BX}{$B_X(x_0, 1)$}
\psfrag{BS}{$B_{\mathbb{S}^{n+1}(R)}(x_1, 1)$} \psfrag{2}{$2$} \psfrag{3}{$3$}
\par
\begin{center}
\includegraphics{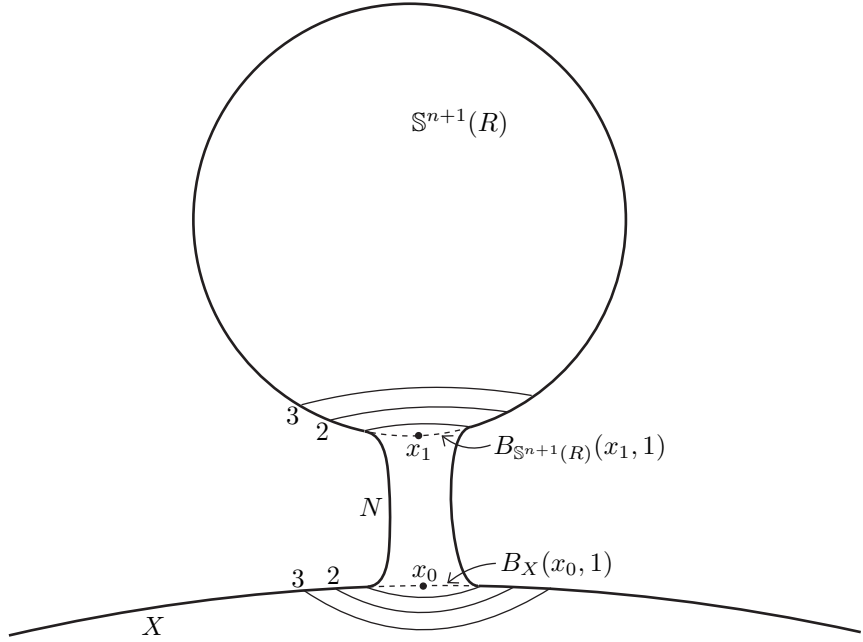}
\end{center}
\caption{$X_{R}$ is constructed by gluing a sphere of radius $R$ to
$X\setminus B_{X}(x_{0},1)$.}%
\label{Xsphere}%
\end{figure}

With this choice of perturbation, we wish to show that $\Phi(\lambda)$ has
essentially the same behavior as the wave trace on the sphere. We now make the
choice $T=2\pi R$ to localize near the periods of geodesics on the sphere. Let
$U_{S}(t)$ denote the wave group on $\mathbb{S}^{n+1}(R)$, and define%
\begin{equation}
\Phi_{0}(\lambda)=\int\varphi_{\varepsilon,2\pi R}(t)\operatorname*{Tr}\left[
U_{S}(t)\right]  ~dt \label{eq:PhiLambda0}%
\end{equation}
Recall (see for example \cite{DG:1975}, section 3):

\begin{lemma}
\label{lemma:wt.sphere}There is a strictly positive constant $c_{n}$ depending
only on $n$ so that%
\[
\Phi_{0}(\lambda)=c_{n}R^{n}\lambda^{n}+\mathcal{O}(\lambda^{n-1}).
\]

\end{lemma}

\begin{proof}
This follows from the fact that the leading singularity of $U_{S}(t)$ at
$t=2\pi R$ is $c_{n}R^{n}\delta^{(n)}(t-2\pi R)$
\end{proof}

We would like to show that $\Phi(\lambda)$ behaves like $\Phi_{0}(\lambda)$ up
to terms of order $R\lambda^{n}$ or lower. Microlocally, $U(t)$ and $U_{S}(t)$
behave similarly except on geodesics that enter the neck region that connects
the sphere to the rest of $X$. To isolate these errors we first define
pseudodifferential operators on the sphere that microlocalize along such
geodesics, and then move them to $(X,g_{1})$. This will allow us to estimate
$\Phi_{1}(\lambda)-\Phi_{0}(\lambda)$.

Let $\widetilde{B}\in\Psi_{\mathrm{phg}}^{0}(\mathbb{S}^{n+1}(R))$ be chosen
so that%
\[%
\operatorname{S-ES}%
(\widetilde{B})\subset\left\{  (x,\xi)\in S^{\ast}\mathbb{S}^{n+1}%
(R):\Pi\gamma_{t}(x,\xi)\in B_{\mathbb{S}^{n+1}(R)}(x_{1},3)~\exists
t\in\mathbb{R}\right\}  ,
\]
and if $\widetilde{A}=I-\widetilde{B}$,%
\[%
\operatorname{S-ES}%
\left(  \widetilde{A}\right)  \subset\left\{  (x,\xi)\in S^{\ast}%
\mathbb{S}^{n+1}(R):\Pi\gamma_{t}(x,\xi)\notin B_{\mathbb{S}^{n+1}(R)}%
(x_{1},11/4)~\forall t\in\mathbb{R}\right\}  .
\]
Note that, here, $\gamma_{t}(x,\xi)$ is a geodesic on the sphere. By adding
smoothing operators if needed, we further require that:

\begin{itemize}
\item $\widetilde{A}f=0$ for all $f\in L^{2}(\mathbb{S}^{n+1}(R))$ with
support in $B_{\mathbb{S}^{n+1}(R)}(x_{1},5/2)$, and

\item $\operatorname*{supp}(\widetilde{A}g)$ is contained $\mathbb{S}%
^{n+1}(R)\backslash B_{\mathbb{S}^{n+1}(R)}(x_{1},5/2)$ for all $g\in
L^{2}(\mathbb{S}^{n+1}(R))$.
\end{itemize}

Next, we define pseudodifferential operators on $X_{R}$ as follows. Let
$\psi_{1}\in\mathcal{C}^{\infty}(\mathbb{S}^{n+1})$ with
\begin{equation}
\psi_{1}(x)=\left\{
\begin{array}
[c]{ccc}%
1 &  & \operatorname*{dist}(x,x_{1})>5/2,\\
&  & \\
0 &  & \operatorname*{dist}(x,x_{1})<9/4,
\end{array}
\right.  \label{eq:cutoff-psi1}%
\end{equation}
and extend by zero to a smooth, compactly supported function on $X_{R}$ which
we continue to denote by $\psi_{1}$. We then define%
\begin{align*}
A  &  =\widetilde{A}\psi_{1},\\
B  &  =I-A.
\end{align*}
Thus $A$ microlocalizes in $S^{\ast}X_{R}$ to trajectories that enter the
gluing region at some time, and $B$ microlocalizes to those that do not.

We now write%
\begin{align*}
\operatorname*{Tr}\left(  U(t)\psi\right)   &  =\operatorname*{Tr}(U_{S}(t))\\
&  +\left[  \operatorname*{Tr}\left(  U(t)A\psi\right)  -\operatorname*{Tr}%
\left(  U_{S}(t)\widetilde{A}\right)  \right] \\
&  -\operatorname*{Tr}\left(  U_{S}(t)\widetilde{B}\right) \\
&  +\operatorname*{Tr}(U(t)B\psi)\\
&  =T_{0}(t)+T_{1}(t)+T_{2}(t)+T_{3}(t)
\end{align*}
and we will set%
\[
\Phi_{i}(\lambda)=\int\varphi_{\varepsilon,2\pi R}(t)\left[  T_{i}(t)\right]
~dt
\]
for $i=0,1,2,3$. \ Note that traces involving $U(t)$ are taken in
$\mathcal{H}(X_{R})$ while those involving $U_{S}(t)$ are taken in
$\mathcal{H}(\mathbb{S}^{n+1}(R))$.

To see that $\Phi_{2}(\lambda)$ and $\Phi_{3}(\lambda)$ give $\mathcal{O}%
(R\lambda^{n})$ contributions we need a phase space estimate.

\begin{lemma}
\label{lemma:phasespace1}The estimate
\begin{equation}
\operatorname*{vol}\nolimits_{S^{\ast}\mathbb{S}^{n+1}(R)}\left(
\operatorname{S-ES}%
(\widetilde{B})\right)  =\mathcal{O}(R) \label{eq:phasevol1}%
\end{equation}
holds.
\end{lemma}

\begin{proof}
Suppose that $\gamma_{t}(x,\xi)$ enters the cap $B_{\mathbb{S}^{n+1}(R)}%
(x_{1},3)$ at some time $t\in\mathbb{R}$. Since the geodesic flow has unit
speed and the closed geodesics have length $2\pi R$, it will enter first at a
time $t\in\lbrack0,2\pi R]$. The volume of the cap $B_{\mathbb{S}^{n+1}%
(R)}(x_{1},3)$ is of order one. Since phase space volume is preserved by
geodesic flow, the phase space volume of points entering the cap, and hence of
$%
\operatorname{S-ES}%
(\widetilde{B})$, is of order $\mathcal{O}(R)$.
\end{proof}

\begin{remark}
\label{rem:phasespace1}The same estimate holds true for $\operatorname*{vol}%
\nolimits_{S^{\ast}X_{R}}\left(
\operatorname{S-ES}%
(B\psi)\right)  $ by construction.
\end{remark}

Combining Lemma \ref{lemma:phasespace1}, Remark \ref{rem:phasespace1}, and
Lemma \ref{lemma.2}, we immediately obtain:

\begin{lemma}
\label{lemma:wt3}The estimate%
\[
\Phi_{2}(\lambda)+\Phi_{3}(\lambda)=\mathcal{O}(R\lambda^{n})
\]
holds.
\end{lemma}

Finally, we prove:

\begin{lemma}
\label{lemma:wt4}The estimate $\Phi_{1}(\lambda)=\mathcal{O}(\lambda^{-\infty
})$ holds.
\end{lemma}

\begin{proof}
First, by the definitions (\ref{eq:cutoff-psi}) and (\ref{eq:cutoff-psi1}) of
$\psi_{1}$ and $\psi$, it follows that $U(t)A\psi=U(t)A$. Next, note that

\begin{itemize}
\item[(i)] if $\widetilde{f}\in L^{2}(X_{R})$ and $\operatorname*{supp}%
\widetilde{f}\subset X_{R}\backslash\left(  \mathbb{S}^{n+1}\backslash
B_{\mathbb{S}^{n+1}(R)}(x_{1},5/2)\right)  $, we have $\widetilde{A}%
\widetilde{f}=0$, and,

\item[(ii)] if $f\in L^{2}(\mathbb{S}^{n+1})$ and $\operatorname*{supp}%
f\subset\mathbb{S}^{n+1}(R)\backslash B_{\mathbb{S}^{n+1}(R)}(x_{1},5/2)$, $f$
has a natural identification with $\tilde{f}\in L^{2}(X_{R})$ and%
\[
Af=\widetilde{A}\widetilde{f}.
\]

\end{itemize}

It follows that $\operatorname*{Tr}(U_{S}(t)\widetilde{A})=\operatorname*{Tr}%
(\psi_{1}U_{S}(t)\widetilde{A})$ and similarly $\operatorname*{Tr}%
(U(t)A)=\operatorname*{Tr}(\psi_{1}U(t)A)$. Moreover,
\[
\operatorname{Tr}_{\mathcal{H}(\mathbb{S}^{n+1}(R))}(\psi U_{S}(t)\widetilde
{A})=\operatorname{Tr}_{\mathcal{H}(\mathbb{S}^{n+1}(R))}(\psi U_{S}(t)A)
\]
if we regard $U_{S}(t)$ as acting on the image of $L^{2}(X_{R})$ under $A$.
Hence $T_{1}(t)=\operatorname*{Tr}G_{3}(t)$ where%
\[
G_{3}(t)=\psi_{1}U(t)A-\psi_{1}U_{S}(t)A.
\]
It suffices to show that $G_{3}(t)$ is a smoothing operator for all $t$. We
have $G_{3}(0)=0$, while%
\[
\left(  \partial_{t}-Q\right)  G_{3}(t)=F_{3}(t)
\]
(recall $Q$ is the generator of $U(t)$) where%
\begin{equation}
F_{3}(t)=\left[  \psi_{1},Q\right]  U(t)A-\left[  \psi_{1},Q\right]  U_{S}(t)A
\label{eq:F3}%
\end{equation}
since the generators of $U(t)$ and $U_{S}(t)$ coincide in the support of
$\psi_{1}$. Since, then%
\begin{equation}
G_{3}(t)=\int_{0}^{t}U(t-s)F_{3}(s)~ds, \label{eq:G3-Duhamel}%
\end{equation}
it is enough to show that the two right-hand terms in (\ref{eq:F3}) are
smoothing operators. By propagation of singularities, the operators $\eta
U(t)A$ and $\eta U_{S}(t)A$ are smoothing for any $\eta\in\mathcal{C}%
_{0}^{\infty}(X_{R})$ vanishing for $x$ with $\operatorname*{dist}%
(x,x_{1})\geq11/4$. Since the commutators $\left[  Q,\psi_{1}\right]  $ and
$\left[  Q_{S},\psi_{1}\right]  $ are supported in $\left\{
x:9/4<\operatorname*{dist}(x,x_{1})<5/2\right\}  $, it follows that $F_{3}(t)$
is smoothing for each $t$, and hence, by (\ref{eq:G3-Duhamel}), $G_{3}(t)$ is
a smoothing operator.
\end{proof}

Collecting Lemmas \ref{lemma:wt3}, \ref{lemma:wt4}, and \ref{lemma:wt.sphere},
we conclude:

\begin{proposition}
The asymptotic formula%
\begin{equation}
\Phi(\lambda)=c_{n}R^{n}\lambda^{n}+\mathcal{O}_{\varepsilon,\psi}%
(R\lambda^{n}) \label{eq:PhiLambda.asy}%
\end{equation}
holds.
\end{proposition}

\begin{proof}
[Proof of Theorem \ref{thm:tumor}]Let $\mathcal{R}_{1}$ be the set of
resolvent resonances for the metric $g_{1}$, and let $u_{1}(t)$ be the
distribution defined in (\ref{eq:u1}). The bound (\ref{eq:0small}) for the
distribution $u_{0}$ and the asymptotic formula (\ref{eq:PhiLambda.asy}) imply
that for $R$ sufficiently large and some strictly positive constant $b$,%
\[
\left\vert \widehat{\varphi_{\varepsilon,2\pi R}~u_{1}}(\lambda)\right\vert
\geq\left(  b-o(1)\right)  \lambda^{n}%
\]
as $\lambda\rightarrow+\infty$. We now apply Theorem \ref{thm:tauber} to
obtain the conclusion.
\end{proof}


\section{Generic lower bounds}

\label{sec:generic1}

We fix a compact region $K\subset X$ and we assume that the metric on
$X\backslash K^{\prime}$ is hyperbolic for some compact region $K^{\prime
}\subset X$ containing $K$. Our goal is to prove that there is a dense
$G_{\delta}$ set $\mathcal{M}(g_{0},K)\subset\mathcal{G}(g_{0},K)$
of metric perturbations for which $N_{g}(r)$, the resolvent resonance counting
function for the perturbed metric has maximal order of growth $n+1$. By the
explicit construction in section \ref{sec:tumor}, the set $\mathcal{M}%
(g_{0},K)$ is nonempty.
We follow the ideas of \cite{Christiansen:2007} and present the main lines of
the argument here. We refer to \cite{Christiansen:2005} and
\cite{Christiansen:2007} for the proofs of statements below that hold with
only minor modification in the present context.


\subsection{Nevanlinna characteristic functions}

\label{subsec:characteristic1}

We recall briefly the main ideas of \cite{Christiansen:2007}. Let $f$ be a
function meromorphic of $\mathbb{C}$. For $r \geq0$, let $n(r,f)$ be the
number of poles of $f$, including multiplicity, in the region $\{
s\in\mathbb{C}:\; |s-n/2|\leq r\}$. We define an integrated counting function
\begin{equation}
\label{eq:counting1}N(r,f) \equiv\int_{0}^{r} [ n(t,f) - n(0, f)] \frac{dt}{t}
+ n(0,f) \log r .
\end{equation}
We also need an average of $\log^{+} |f|$ along the contour $|s-n/2| = r$:
\begin{equation}
\label{eq:semicircle1}m(r,f) \equiv\frac{1}{2 \pi} \int_{0}^{2 \pi} \log^{+}
|f(n/2+ r e^{i \theta})| ~d \theta,
\end{equation}
where $\log^{+} (a) = \mbox{max} ~(0 , \log a)$, for $a > 0$. The
\textit{Nevanlinna characteristic function}\footnote{Strictly speaking, this
is the Nevanlinna characteristic function of $f(s+n/2)$, rather than that for
$f$. We have chosen to make this minor adaptation here to suit the importance
of $s=n/2$ in our parameterization of the spectrum.} of $f$ is defined by
\begin{equation}
\label{eq:nevanlinna1}T(r,f ) \equiv N(r,f) + m(r,f) .
\end{equation}
This is a nondecreasing function of $r$. The \textit{order} of a
nondecreasing, nonnegative function $h(r) > 0$ is given by
\begin{equation}
\label{eq:order1}\limsup_{r \rightarrow\infty} \frac{\log h(r)}{\log r} = \mu,
\end{equation}
provided it is finite. The order of a meromorphic function $f$ is the order of
its characteristic function $T(r,f)$.

The following proposition gives a connection between the order of the
characteristic function of $f$ and the order of the pole counting function
$n(r,f)$ for $f$ under certain conditions on the meromorphic function $f$.
We recall this result from \cite[Lemma 2.3]{Christiansen:2007} (see also
\cite[Lemma 4.2]{Christiansen:2005}) with minor changes to suit the convention
that the right half-plane $\Re(s) >n/2$ corresponds to the physical region.

\begin{proposition}
\label{prop:countingfunc1} Suppose that $f(s)$ is a meromorphic function on
$\mathbb{C}$ with the property that $s_{0}$ is a pole of $f$ if and only if
$n- s_{0}$ is a zero of $f$, and the multiplicities are the same. Furthermore,
suppose that no zeros of $f$ lie on the line $\Re(s)=n/2$ and that
\begin{equation}
\label{eq:line1}\int_{0}^{r} \frac{d}{dt} \log f(n/2+it) ~dt = \mathcal{O}
(r^{m}),
\end{equation}
for some $m > 1$. Then, $f$ is of order $p > m$ if and only if $n(r,f)$ is of
order $p$.
\end{proposition}

We next introduce the auxiliary parameter $z$ taking values in an open
connected set $\Omega\subset\mathbb{C}$. We consider functions $f(z,s)$ that
are meromorphic on $\Omega_{z}\times\mathbb{C}_{s}$. Considering $z\in\Omega$
as a parameter, we write $T(z,r,f)\equiv T(r,f(z,\cdot))$ for the Nevanlinna
characteristic function of $f(z,s)$.

For any $z_{0} \in\Omega$, let $\Omega_{0} \subset\Omega$ denote an open ball
centered at $z_{0}$. Given $z_{0} \in\Omega_{0}$, there are holomorphic,
relatively prime functions $g_{\Omega_{0}}$ and $h_{\Omega_{0}}$ defined on
$\Omega_{0} \times\mathbb{C}$, so that
\begin{equation}
\label{eq:fraction1}f(z, s) = \frac{ g_{\Omega_{0}} (z, s)}{ h_{\Omega_{0}} (z
, s)}, ~~\mbox{for} ~~(z , s) \in\Omega_{0} \times\mathbb{C}.
\end{equation}
We suppose that $h_{\Omega_{0}}(z , s )= (s-n/2)^{j} \tilde{h}_{\Omega_{0}}(z
, s)$ so that $\tilde{h}_{\Omega_{0}}$ is holomorphic on $\Omega_{0}
\times\mathbb{C}$ and $\tilde{h}_{0} (z , n/2)$ is not identically zero. We
define a set $K_{f, \Omega_{0}}$ relative to this decomposition by
\begin{equation}
\label{eq:fraction2}K_{f, \Omega_{0}} = \{ z_{1} \in\Omega_{0} ~|~ \tilde
{h}_{\Omega_{0}} ( z_{1} , n/2) = 0 ~\mbox{or} ~ h_{\Omega_{0}}(z_{1} , s)
~\mbox{vanishes identically}, s \in\mathbb{C} \}.
\end{equation}
The set $K_{f, \Omega_{0}}$ is independent of the decomposition described
above provided each pair $(g_{\Omega_{0}},h_{\Omega_{0}})$ satisfies the same
properties. We let $K_{f}$ be the union of all these sets over balls
$\Omega_{0}$ for each $z_{0} \in\Omega$. The intersection of $K_{f}$ with any
compact subset of $\Omega$ consists of a finite number of points.

The next result illustrates the utility of the additional parameter $z$. If
the order of the monotone nondecreasing function $r\mapsto T(z,r,f)$ is
bounded and the bound is obtained at some $z_{0}\in\Omega\backslash K_{f}$
then it is obtained at all points $z\in\Omega\backslash K_{F}$ except for a
pluripolar set. For the definition of pluripolar sets and additional facts
about them see, for example \cite{LG:1986} or \cite{Klimek:1991}. Pluripolar
sets are small. In particular, we shall use the fact that if $\Omega
\subset\mathbb{C}$ is open and $E\subset\Omega$ is pluripolar, then
$\Omega\cap\mathbb{R}$ has Lebesgue measure zero.

\begin{theorem}
\cite[Theorem 3.5]{Christiansen:2007}\label{th:psh1} Let $\Omega
\subset\mathbb{C}$ be an open connected set. Let $f(z,s)$ be meromorphic on
$\Omega_{z} \times\mathbb{C}_{s}$. Suppose that the order $\rho(z)$ of the
function $r \mapsto T(r,f(z, \cdot))$ is at most $\rho_{0}$ for $z \in
\Omega\backslash K_{f}$, and that there is a point $z_{0} \in\Omega\backslash
K_{f}$ such that $\rho(z_{0} ) = \rho_{0}$. Then, there exists a pluripolar
set $E \subset\Omega\backslash K_{f}$ such that $\rho(z) = \rho_{0}$ for all
$z \in\Omega\backslash( E \cup K_{f})$.
\end{theorem}

It follows by Proposition \ref{prop:countingfunc1} that the order of the pole
counting function for $f$, $n(r,f(z, \cdot))$, is the same order $\rho_{0}$
for $z \in\Omega\backslash( E \cup K_{f})$ provided condition (\ref{eq:line1})
and the other hypotheses are satisfied.

\subsection{Density of $\mathcal{M} (g_{0}, K)$}

\label{subsec:density1}

In this subsection we prove

\begin{proposition}
\label{prop:density} The set $\mathcal{M} (g_{0}, K) \subset\mathcal{G}(
g_{0}, K)$ is dense in the $C^{\infty}$ topology.
\end{proposition}

To do this, we need to show that given a metric $\tilde{g}\in\mathcal{G}(
g_{0}, K)$ there is a sequence of metrics in $\mathcal{M} (g_{0}, K)$
approaching $\tilde{g}$ in the $C^{\infty}$ topology. If $\tilde{g}%
\in\mathcal{M} (g_{0}, K)$, we are, of course, done. If not, noting that
$\mathcal{M} (g_{0}, K)= \mathcal{M} (\tilde{g}, K)$ and $\mathcal{G}( g_{0},
K)= \mathcal{G}( \tilde{g}, K)$, we may (by relabeling) reduce the problem to
assuming that $g_{0}$ itself is resonance-deficient, and finding a sequence of
metrics in $\mathcal{M} (g_{0}, K)$ approaching $g_{0}$. In what follows let
\[
\sigma_{g,g_{0}}(s)=\det[ S_{g}(s) S_{g_{0}}^{-1} (s)].
\]

As in section \ref{sec:complexmetric1} we consider a complex interpolation
between a smooth metric $g_{0}$ that is hyperbolic outside a compact
$K^{\prime}\subset X$, and a metric $g_{1}\in\mathcal{M}(g_{0},K)$. The
existence of such a metric $g_{1}$ is precisely the result of section
\ref{sec:tumor}. As in (\ref{eq:complex1}), this interpolated
\textquotedblleft metric\textquotedblright\ is given by $g_{z}=(1-z)g_{0}%
+zg_{1}$, where $z\in\Omega_{\epsilon}$ with $\Omega_{\epsilon}$ as in
(\ref{eq:compldomain1}).

The scattering matrix $S_{g_{z}}(s)$ is defined in section
\ref{sec:complexmetric1} along with the corresponding relative scattering
phase. We define a relative volume factor (see Corollary \ref{weyl.asymp}) by
\begin{align}
V_{\mathrm{rel}}(z)  &  \equiv\Delta\operatorname{Vol}(g_{z},g_{0}%
)\label{eq:relvol1}\\
&  =\int_{K}(\sqrt{\det(g_{z})}-\sqrt{\det(g_{0})})\nonumber\\
&  =\operatorname{Vol}(K,g_{z})-\operatorname{Vol}(K,g_{0}),\nonumber
\end{align}
and note that this is analytic in $z$ in a possibly smaller region that we
still call $\Omega_{\epsilon}$. With $c_{n}$ the constant from Corollary
\ref{weyl.asymp}, we shall use the function
\begin{equation}
f(z,s)=e^{-c_{n}V_{\mathrm{rel}}(z)(-is)^{n+1}}\sigma_{g_{z},g_{0}}(s),
\label{eq:basicfnc1}%
\end{equation}
meromorphic in $(z,s)\in\Omega_{\epsilon}\times\mathbb{C}$.

First, we note from Proposition \ref{detsrel.factor} that if $s_{0}$ is a pole
of $f$ then $n-s_{0}$ is a zero of $f$ and the multiplicities coincide.
Second, using Corollary \ref{weyl.asymp}, we find that for $a\in\mathbb{R}$
and $t\rightarrow\infty$,
\begin{align}
\log f(a,n/2+it)  &  =\log\sigma_{g_{a},g_{0}}(n/2+it)-c_{n}V_{\mathrm{rel}%
}(a)t^{n+1}+O(t^{n})\label{eq:weyl.asymp2}\\
&  =\mathcal{O}(t^{n}).\nonumber
\end{align}
Consequently, hypothesis (\ref{eq:line1}) is
\begin{equation}
\int_{0}^{r}\frac{d}{dt}\log f(n/2+it)~dt=\mathcal{O}(r^{n}).
\label{eq:scatph1}%
\end{equation}
Hence, from Proposition \ref{prop:countingfunc1}, if can can prove that
$f(z,s)$ is order $n+1$ for a large set of $z\in\Omega_{z}$, it will follow
that the corresponding resonance counting function is order $n+1$ for the same
set of $z$.

To this end, we appeal to Theorem \ref{th:psh1}. We know from section
\ref{sec:tumor} that $f(1,s)$ has the correct order of growth $n+1$.
Furthermore, we note the following bound, which follows directly from
Proposition \ref{detsrel.factor}.

\begin{lemma}
\label{l:phase-order1} The order of the function $s \mapsto f(z,s)$ is at most
$n+1$ for $z \in\Omega_{\epsilon}\backslash K_{f}$.
\end{lemma}

To apply Theorem \ref{th:psh1} we need, in addition, that $z=1$ is not in
$K_{f}$. This may, in fact, fail. But if $1\in K_{f}$, we may consider instead
the function $f_{1}(z,s)=f(z,s+i)$. Then $z=1$ is not in $K_{f_{1}}$, because
$n/2+i$ is not a pole of $R_{g_{1}}(s)$. Thus we may first apply Theorem
\ref{th:psh1} to $f_{1}$, and then apply Proposition \ref{prop:countingfunc1}
to $f$, noting that $s\mapsto f_{1}(z,s)$ and $s\mapsto f(z,s)$ have the same
order. From Theorem \ref{th:psh1}, there exists a pluripolar set
$E\subset\Omega$ so that for all $z\in\Omega_{\epsilon}\backslash(K_{f}\cup
E)$, the resonance counting function has optimal order of growth. Since
$(K_{f}\cup E)\cap\mathbb{R}$ has Lebesgue measure $0$, there is a sequence of
real $\lambda_{j}\downarrow0$ so that $N_{g_{\lambda_{j}}}(r)$ has maximal
order of growth. Then, for any $\epsilon>0$ there is a $J(\epsilon)$ so that
the metric $g_{\lambda_{j}}$ satisfies $d_{\infty}(g_{\lambda_{j}}%
,g_{0})<\epsilon$ whenever $j>J(\epsilon)$. This finishes the proof of
Proposition \ref{prop:density}.

\subsection{The $G_{\delta}$-Property of $\mathcal{M}( g_{0},K)$}

\label{subsec:generic1}

The main result of this subsection is:

\begin{proposition}
\label{prop:Gdelta} The set $\mathcal{M} (g_{0}, K) \subset\mathcal{G}( g_{0},
K)$ is a $G_{\delta}$ set.
\end{proposition}

If $\mathcal{M}(g_{0},K)=\mathcal{G}(g_{0},K)$, meaning there are no
resonance-deficient metrics in $\mathcal{G}(g_{0},K)$, then there is nothing
to prove. So suppose there is a resonance-deficient metric $g\in
\mathcal{G}(g_{0},K)$. Since $\mathcal{M}(g_{0},K)=\mathcal{M}(g,K)$, and
$\mathcal{G}(g_{0},K)=\mathcal{G}(g,K)$, we may, as before, assume $g_{0}$
itself is resonance- deficient.

Define, for any $g\in\mathcal{G}(g_{0},K)$, $r>0$,
\begin{multline*}
h_{g}(r)=\frac{1}{2\pi i}\int_{0}^{r}t^{-1}\int_{-t}^{t}\frac{\sigma_{g,g_{0}%
}^{\prime}(n/2+i\tau)}{\sigma_{g,g_{0}}(n/2+i\tau)}d\tau dt\\
+\frac{1}{2\pi}\int_{-\pi/2}^{\pi/2}\log\left\vert \sigma_{g,g_{0}%
}(n/2+re^{i\theta})\right\vert d\theta.
\end{multline*}
This function is useful because of the following

\begin{lemma}
\label{l:ordercomp} If $\lim\sup_{r\rightarrow\infty}\dfrac{\log N_{g_{0}}%
(r)}{\log r}=p<n+1$, and $p^{\prime}>p,$ then
\[
\lim\sup_{r\rightarrow\infty}\frac{\log[\max(h_{g}(r),1)]}{\log r}=p^{\prime}%
\]
if and only if $N_{g}(r)$ has order $p^{\prime}$.
\end{lemma}

\begin{proof}

Let $f$ be meromorphic in a neighborhood of the closed half plane $\{ s:
\Re(s)\geq n/2\}$, and such that $f$ has neither zeros nor poles on the line
$\Re(s)=n/2$. Let $Z_{f}(r) =\int_{0}^{r} t^{-1} n_{f,Z}(t) dt$ where
$n_{f,Z}(r)$ is the number of zeros of $f(s)$ in $\{ s: \Re(s)> n/2, \;
|s-n/2|\leq r\}$, and define $P_{f,r}$ analogously as counting the poles of
$f$ in the same region. Then
\begin{align}
\label{eq:zeros-poles}Z_{f}(r)-P_{f}(r)  &  = \frac{1}{2\pi} \Im\int_{0}^{r}
t^{-1} \int_{-t}^{t} \frac{ f^{\prime}(n/2+i\tau)}{f(n/2+i\tau)} d\tau d t\\
&  + \frac{1}{2 \pi} \int_{-\pi/2}^{\pi/2} \log|f(n/2 +r e^{i\theta}%
)|d\theta.\nonumber
\end{align}
This identity follows essentially exactly as the proof of \cite[Lemma
6.1]{froese}, the primary difference being the application of the argument
principle for meromorphic, rather than holomorphic, functions.

For $\Re(s_{0})>n/2$, if $s_{0}$ is a pole of order $k$ of $\sigma_{g,g_{0}%
}(s)$, set $\mu_{\text{rel}}(s_{0})=-k$; otherwise, set $\mu_{\text{rel}%
}(s_{0})$ to be the order of the zero of $\sigma_{g,g_{0}}(s)$ at $s_{0}$ (of
course, $\mu_{\text{rel}}(s_{0})=0$ if $s_{0}$ is neither a zero nor a pole).
Now we use again, as follows from Proposition \ref{detsrel.factor} that for
$\Re(s)>n/2$,
\begin{equation}
\label{eq:relativemult}\mu_{rel}(s)= m_{g}(n-s)-m_{g}(s)-m_{g_{0}%
}(n-s)+m_{g_{0}}(s)
\end{equation}
where $m_{g}$ (resp., $m_{g_{0}}$) is as defined in (\ref{eq:multiplicity1})
for the metric $g$ (resp. $g_{0}$).

In the notation of (\ref{eq:zeros-poles}), the order of $P_{\sigma_{g,g_{0}}%
}(r)$ is at most $p$, the order of the resonance counting function for
$\Delta_{g_{0}}$. Thus, using (\ref{eq:zeros-poles}),
\[
\lim\sup_{r\rightarrow\infty}\left(  \frac{\log[\max(h_{g}(r),1)]}{\log
r}\right)  =p^{\prime}>p
\]
if and only if the order of $Z_{\sigma_{g,g_{0}}}(r)$ is $p^{\prime}$. The
order of $Z_{\sigma_{g,g_{0}}}(r)$ is the same as the order of $n_{\sigma
_{g,g_{0}},Z}(r)$. Using (\ref{eq:relativemult}) and the fact that $N_{g_{0}%
}(r)$ has order $p$, the order of $n_{\sigma_{g,g_{0}},Z}(r)$ is $p^{\prime
}>p$ if and only if the order of $N_{g}(r)$ is $p^{\prime}$.
\end{proof}

Define, for $M,\;q,\;j,\;\alpha>0$, the set
\begin{multline*}
A(M,q,j,\alpha)=%
\bigl\{%
g\in\mathcal{G}(g_{0},K):\;\\
\sum_{i,l}g^{il}\xi_{i}\xi_{l}\geq\alpha|\xi|^{2}\text{ on }K,h_{g}(r)\leq
M(1+r^{q})\text{ for }0\leq r\leq j%
\bigr\}%
.
\end{multline*}

\begin{lemma}
For $M,\; q,\; j,\;\alpha>0$, the set $A(M, q,j,\alpha)$ is closed.
\end{lemma}

\begin{proof}
Let $g_{m}\in A(M, q,j,\alpha)$ be a sequence of metrics converging in the
$C^{\infty}$ topology. Since $\sum_{i,j}g_{m}^{ij}\xi_{i} \xi_{j} \geq
\alpha|\xi|^{2} $, $\{g_{m}\}$ converges to a metric $g$ with the same property.

Since $g_{m}\rightarrow g$ in the $C^{\infty}$ topology, we also have
convergence of the cut-off resolvents: for $\chi\in C_{c}^{\infty}(X)$, $\chi
R_{g_{m}}(s)\chi\rightarrow\chi R_{g}(s)\chi$ for values of $s$ for which
$\chi R_{g}(s)\chi$ is a bounded operator. This includes the closed half plane
$\{\Re(s)\geq n/2\}$ with the possible exception of a finite number of points
corresponding to the discrete spectrum. Thus using the equations (\ref{Ez.E0})
and (\ref{Sz.S0}) for the scattering matrix, we see that if $\Re(s_{0})\geq
n/2$, $S_{g_{0}}$ has no null space at $s_{0}$, $\Re(s_{0})>n/2$, and
$s_{0}(n-s_{0})$ is not an eigenvalue of $\Delta_{g}$, then $S_{g_{m}%
}(s)S_{g_{0}}^{-1}(s_{0})\rightarrow S_{g}(s_{0})S_{g_{0}}^{-1}(s_{0})$ in the
trace class norm. This convergence is uniform on compact sets which include no
poles of either $S_{g_{0}}^{-1}$ or of $R_{g}$. Thus, if the set
$\{s:\Re(s)>n/2,\;|s-n/2|=r\}$ contains no zeros of $S_{g_{0}}(s)$ or of
$S_{g}(s)$, then $h_{g_{m}}(r)\rightarrow h_{g}(r)$. Thus $h_{g_{m}%
}(r)\rightarrow h_{g}(r)$ for all but a discrete set of values of $r$ in
$[0,j]$. Since $h_{g}(r)$ and $h_{g_{m}}(r)$ are continuous, we get the
desired upper bound on $h_{g}(r)$ for all $r\in\lbrack0,j]$.
\end{proof}

Now, for $M,\; q,\; \alpha>0$, set
\[
B(M,q,\alpha)= \cap_{j\in{\mathbb{N}}} A(M,q,j,\alpha).
\]
The set $B(M,q,\alpha)$ is closed since $A(M,q,j,\alpha)$ is closed. The proof
of Proposition \ref{prop:Gdelta} is completed by the following lemma.

\begin{lemma}
If $g_{0}$ is resonance-deficient, then
\[
\mathcal{G}(g_{0},K)\setminus\mathcal{M}(g_{0},K)=\cup_{(M,l,m)\in{\mathbb{N}%
}^{3}}B(M,n+1-1/l,1/m).
\]

\end{lemma}

\begin{proof}
If $g\in B(M,n+1-1/l,1/m)$ for some $M,\;l,\;m>0$, then by Lemma
\ref{l:ordercomp} the order of growth of $N_{g}(r)$ is at most the maximum of
$n+1-1/l$ and the order of growth of the resonance counting function of
$N_{g_{0}}$, so $g\not \in \mathcal{M}(g_{0},K).$

Suppose $g\in\mathcal{G}(g_{0},K)\setminus\mathcal{M}(g_{0},K)$. Then the
order of $N_{g}(r)$ is $p^{\prime}$ for some $p^{\prime}<n+1$. An application
of Lemma \ref{l:ordercomp} shows that there are integers $M$ and $l$ so that
$p^{\prime}<n+1-1/l<n+1$ and $g\in B(M,n+1-1/l,\alpha)$ for some $\alpha>0$
sufficiently small.
\end{proof}

\bigskip

\begin{proof}
[Proof of part (ii) of Theorem \ref{thm:main}:]This is immediate from
Propositions \ref{prop:density} and \ref{prop:Gdelta}.
\end{proof}

\appendix

\section{Estimates for the wave trace}

\label{app:wavetrace1}

In this appendix we prove a key lemma , essentially taken from
Sj\"{o}strand-Zworski \cite{SZ:1993}, which plays an important role in section
\ref{sec:tumor}. To formulate the statement, recall that $\Psi_{\mathrm{phg}%
}^{m}(M)$ denotes the polyhomogeneous pseudodifferential operators of order
$m$ on $M$. For $P\in\Psi_{\mathrm{phg}}^{m}(X)$, we recall that the
\emph{essential support} of $P$, denoted $\operatorname{ES} (P)$, as follows.
For a conic open subset $U$ of $T^{\ast}M$, we say that $P$ has order
$-\infty$ on $U$ if $\left\vert p(x,\xi)\right\vert \leq C_{N}\left(
1+\left\vert \xi\right\vert \right)  ^{-N}$ for every $N$ and $\left(
x,\xi\right)  \in U$. The essential support $\operatorname{ES}(P)$ is the
smallest conic subset of $T^{\ast}M$ on the complement of which $P$ has order
$-\infty$ (see for example Taylor \cite[Chapter VI, Definition 1.3]%
{Taylor:1991} for discussion). Note that $\operatorname{ES} (P_{1}%
P_{2})\subset\operatorname{ES} (P_{1})\cap\operatorname{ES} (P_{2})$ by the
usual symbol calculus (see for example Taylor \cite{Taylor:1991}, \S 0.10 for
further discussion). In particular, if $P_{1}$ and $P_{2}$ have disjoint
essential supports, then $P_{1}P_{2}$ is a smoothing operator.

For a pseudodifferential operator $A$, we set%
\[
\operatorname{S-ES}(A)=\operatorname{ES}(A)\cap S^{\ast}M.
\]
We denote by $\operatorname*{dist}_{S^{\ast}M}$ the distance on $S^{\ast}M$
induced by the Riemannian metric on $S^{\ast}M$. Since the essential support
is a conic set these two notions are equivalent. One should think of the
pseudodifferential operators $B$ and $C$ that occur in Lemma \ref{lemma.2} as
smoothed characteristic functions of a small region of $S^{\ast}X$ so that the
operator $C$ has a wave front set slightly bigger than that of $B$ and $B\sim
C^{2}$; compare \cite{SZ:1993}, pp. 854-855. In what follows, $Q$ is a
first-order, self-adjoint, scalar pseudodifferential operator (one should
think of $Q=\sqrt{\Delta-n^{2}/4}$ in the application) and $V(t)=\exp(itQ)$;
thus $Q$ here occurs in the diagonalization of the matrices $Q$ that occur in
section \ref{sec:tumor}).

\begin{lemma}
\label{lemma.2} Let $Q\in OPS_{1,0}^{1}(M)$ and let $B\in\Psi_{\mathrm{phg}%
}^{0}(M)$. Let $\chi\in\mathcal{C}_{0}^{\infty}(\mathbb{R})$ with support near
$t=0$, $\chi(0)\neq0$ and $\widehat{\chi}(t)\geq0$. Let $C$ be a self-adjoint
operator in $\Psi_{\mathrm{phg}}^{0}(X)$ with $\left(  x,\omega\right)
\notin\operatorname{S-ES}(I-C)$ if $\operatorname*{dist}_{S^{\ast}M}\left(
\left(  x,\omega\right)  ,\operatorname{S-ES}(B)\right)  \leq1$ and $\left(
x,\omega\right)  \notin\operatorname{S-ES}(C)$ if $\operatorname*{dist}%
_{S^{\ast}M}((x,\xi),\operatorname{S-ES}(B))\geq2$. Then
\begin{multline}
\left\vert \int e^{-i\lambda t}\chi(t-T)\operatorname*{Tr}\left(
V(t)B\right)  ~dt\right\vert \label{eq:szform}\\
\leq c_{n}~\chi(0)~\left\Vert B\right\Vert \left(  \int_{S^{\ast}X}\left\vert
c(x,\omega)\right\vert ^{2}~dx~d\omega\right)  \lambda^{n}+\mathcal{O}%
_{B,T,\chi}\left(  \lambda^{n-1}\right)  .\nonumber
\end{multline}

\end{lemma}

\begin{proof}
Following Sj\"{o}strand and Zworski \cite{SZ:1993} we set $t=T+s$ and write%
\begin{align*}
M(\lambda)  &  :=\int e^{-i\lambda t}\chi(t-T)\operatorname*{Tr}\left(
V(t)B\right)  ~dt\\
&  =\int e^{-i\lambda T}e^{-i\lambda s}\chi(s)\operatorname*{Tr}\left(
e^{iTQ}e^{isQ}B\right)  ~dt\\
&  =e^{-i\lambda T}\operatorname*{Tr}\left(  e^{iTQ}\widehat{\chi}%
(\lambda-Q)B\right)
\end{align*}
so that%
\[
\left\vert M(\lambda)\right\vert \leq\left\Vert \widehat{\chi}(\lambda
-Q)B\right\Vert _{\mathcal{I}_{1}}%
\]
where we have used the fact that $\left\Vert AB\right\Vert _{\mathcal{I}_{1}%
}\leq\left\Vert A\right\Vert \left\Vert B\right\Vert _{\mathcal{I}_{1}}$ to
eliminate the unitary group $e^{iTQ}$ and reduce to a \textquotedblleft
small-time\textquotedblright\ estimate. Here and in what follows, $\left\Vert
~\cdot~\right\Vert $ denotes the operator norm. For any fixed smoothing
operator $S$, $\left\Vert \widehat{\chi}(\lambda-Q)S\right\Vert =\mathcal{O}%
\left(  \lambda^{-\infty}\right)  $. From the essential support properties of
$B$ and $C$, it is clear that $B\left(  I-C\right)  $ and $(I-C)B$ are
smoothing. Moreover, the operator
\[
\left(  I-C\right)  \widehat{\chi}(\lambda-Q)B=\int\chi(s)e^{-i\lambda
s}\left(  I-C\right)  e^{isQ}B~ds
\]
obeys the estimate%
\[
\left\Vert \left(  I-C\right)  \widehat{\chi}(\lambda-Q)B\right\Vert
_{\mathcal{I}_{1}}\leq\int\left\vert \chi(s)\right\vert \left\Vert \left(
I-C\right)  B(s)\right\Vert _{\mathcal{I}_{1}}~ds
\]
where $B(s):=e^{isQ}Be^{-isQ}$ has wave front set disjoint from
$\operatorname{S-ES}(I-C)$ for small $s$ owing to the support properties of
$C$, so that the trace-norm under the integral is finite. By continuity
$\left\Vert \left(  I-C\right)  B(s)\right\Vert _{\mathcal{I}_{1}}$ is bounded
for small $s$ so that
\[
\left\Vert \left(  I-C\right)  \widehat{\chi}(\lambda-Q)B\right\Vert
_{\mathcal{I}_{1}}\leq C
\]
uniformly in $\lambda$. Hence, we may estimate%
\begin{align*}
\left\vert M(\lambda)\right\vert  &  \leq\left\Vert \left(  I-C\right)
\widehat{\chi}(\lambda-Q)B\right\Vert _{\mathcal{I}_{1}}+\left\Vert
C\widehat{\chi}(\lambda-Q)\left(  I-C\right)  B\right\Vert _{\mathcal{I}_{1}%
}+\left\Vert C\widehat{\chi}(\lambda-Q)CB\right\Vert _{\mathcal{I}_{1}}\\
&  \leq\left\Vert B\right\Vert \left\Vert C\widehat{\chi}(\lambda
-Q)C\right\Vert _{\mathcal{I}_{1}}+\mathcal{O}_{B,T,\chi}\left(  1\right)
\end{align*}
where $\mathcal{O}_{B,T}(1)$ denotes a constant depending on $B$, $T$, and
$\chi$ but independent of $\lambda$. Since $\widehat{\chi}$ is positive and
$C$ is self-adjoint, we have%
\begin{align*}
\left\Vert C\widehat{\chi}(\lambda-Q)C\right\Vert _{\mathcal{I}_{1}}  &
=\operatorname*{Tr}\left(  C\widehat{\chi}(\lambda-Q)C\right) \\
&  =\int e^{-i\lambda s}\chi(s)\operatorname*{Tr}\left(  C^{2}e^{isQ}\right)
~ds.
\end{align*}
We now use H\"{o}rmander's lemma, Lemma \ref{lemma:hormander1} below, to
complete the proof.
\end{proof}

Let $X$ be a compact connected manifold without boundary. H\"ormander's lemma
is the following result and appears as \cite[Proposition 29.1.2]%
{HormanderIV:1985}.

\begin{lemma}
\label{lemma:hormander1} Let $B\in\Psi_{\mathrm{phg}}^{0}(X,\Omega
^{1/2},\Omega^{1/2})$ with principal symbol $b$ and subprincipal symbol
$b^{s}$, and let $P$ have principal symbol $p$ and subprincipal symbol $p^{s}%
$. Let $E(t)$ solve $\left(  D_{t}+P\right)  E(t)=0$ with $E(0)=I$. Let $K$ be
the restriction to the diagonal $\Delta$ of the Schwarz kernel of $E(t)B$.
Then~$K$ is conormal with respect to $\Delta\times\left\{  0\right\}  $ for
$\left\vert t\right\vert $ small and
\begin{equation}
K(t,y)=\int\frac{\partial A(y,\lambda)}{\partial\lambda}e^{-i\lambda
t}~d\lambda, \label{eq:expansion1}%
\end{equation}
where%
\begin{align}
A(y,\lambda)  &  =(2\pi)^{-n}\int_{p(y,\eta)<\lambda}(b+b^{s})(y,\eta
)~d\eta\label{eq:expansion2}\\
&  +\frac{\partial}{\partial\lambda}\int_{p(x,\eta)<\lambda}\left(
p^{s}b+\frac{1}{2}\left\{  b,p\right\}  \right)  ~d\eta\nonumber\\
&  +(S^{n-2})\nonumber
\end{align}
where $S^{n-2}$ means a symbol of order $n-2$ in the $\lambda$ variable.
\end{lemma}

Note that the second integral has lower order so the dominant term gives the
leading singularity. Applying this to our case gives the expected leading behavior.

\end{document}